\documentclass[11pt]{amsart}

\usepackage[mathscr]{euscript}

\usepackage{mathrsfs}

\usepackage{amsmath}
\usepackage[margin=1in]{geometry}
\usepackage{enumerate}

\usepackage{hyperref}

\allowdisplaybreaks

\newtheorem{prop}{Proposition}

\newtheorem{theo}[prop]{Theorem}
\newtheorem{lemma}[prop]{Lemma}
\newtheorem{lemm}[prop]{Lemma}
\newtheorem{coro}[prop]{Corollary}

\newtheorem{claim}[prop]{Claim}

\theoremstyle{definition}
\newtheorem{defi}[prop]{Definition}

\newtheorem{remark}[prop]{Remark}

\newcommand{\CC}{\mathbb{C}}

\newcommand{\NN}{\mathbb{N}}

\newcommand{\RR}{\mathbb{R}}
\renewcommand{\SS}{\mathbb{S}}

\newcommand{\cB}{\mathcal B}

\newcommand{\cL}{\mathcal L}

\newcommand{\cS}{\mathcal S}

\newcommand{\cY}{\mathcal Y}

\newcommand{\sL}{\mathscr{L}}

\DeclareMathOperator{\Span}{span}

\DeclareMathOperator{\Vol}{Vol}
\DeclareMathOperator{\range}{range}

\let\oldmarginpar\marginpar
\renewcommand\marginpar[1]{\-\oldmarginpar[\raggedleft\footnotesize #1]%
{\raggedright\footnotesize #1}}

\DeclareMathOperator{\proj}{proj}
\DeclareMathOperator{\Euc}{Euc}
\DeclareMathOperator{\vol}{vol}

\usepackage{graphicx}

\newcommand{\T}{\mathbin{\raisebox{6 pt}{\scalebox{-.8}{{$\perp$}}}}}

\allowdisplaybreaks

\title{Slowly converging Yamabe flows}

\thanks{This material is based upon work supported in part by NSF grants
DMS-0802923, DMS-1105323, DMS-1206284, and DGE-1147470. OC additionally acknowledges the support of NSF grants DMS 1107452, 1107263, and 1107367 ``RNMS: Geometric structures and Representation varieties" (the GEAR Network). YAR was also supported by a Sloan Research Fellowship. The authors
thank the referee for comments improving the exposition,
R. Bettiol and J.-P. Bourguignon for valuable discussions, R. Mazzeo and F. Weissler
for informing us of \cite{BidautVeronBouhar,CWeissler}, and S. Brendle, 
A. Malchiodi, R. Mazzeo, and R. Schoen for their interest and encouragement. }

\author[A. Carlotto]{Alessandro Carlotto}
\address{Department of Mathematics, Stanford University, Stanford, CA 94305, USA}
\email{carlotto@math.stanford.edu}
\author[O. Chodosh]{Otis Chodosh}
\email{ochodosh@math.stanford.edu}

\author[Y.A. Rubinstein]{Yanir A. Rubinstein}
\address{Department of Mathematics, University of Maryland, College Park, MD 20742, USA}
\email{yanir@umd.edu}

\date{\today}

\begin{document}

\begin{abstract}
We characterize the rate of convergence of a converging volume-normalized Yamabe flow in terms of Morse theoretic properties of the limiting metric. If the limiting metric is an integrable critical point for the Yamabe functional (for example, this holds when the critical point is non-degenerate), then we show that the flow converges exponentially fast. In general, we make use of a 
suitable \L ojasiewicz--Simon inequality to prove that the slowest the flow will converge is polynomially. When the limit metric satisfies an Adams--Simon type condition we prove that there exist flows converging to it exactly at a polynomial rate. We conclude by constructing explicit examples to show that this does occur. These seem to be the first examples of a slowly converging solution to a geometric flow.
\end{abstract}

\maketitle

\section{Introduction}
\label{IntroSec}

Let $M^{n}$ be an arbitrary smooth closed manifold of dimension $n\geq 3$ and set $N=\frac{2n}{n-2}$. In this article we study the quantitative rate of convergence of the volume-normalized Yamabe flow
\begin{equation*}
\frac{\partial g}{\partial t} = -(R_{g} - r_{g})g,
\end{equation*}
for complete Riemannian metrics $g(t)$ on $M$.
Here $R_{g}$ is the scalar curvature and $r_g$ is its average. This is a flow on a volume normalized conformal class on $M$. It arises as the gradient flow of the Einstein--Hilbert functional and thus is a fundamental tool in the study of scalar curvature deformations, mostly in connection with the celebrated Yamabe problem. Motivated by the well-known uniformization theorem, the problem asks whether for any given Riemannian manifold $(M_{0},g_{0})$ one can find a positive function $w$ such that the conformal metric $w^{N-2}g_{0}$ has {constant} scalar curvature. An affirmative answer to this question was obtained by the combined efforts of Yamabe \cite{Yamabe}, Trudinger \cite{Trudinger:confDefRiem}, Aubin \cite{Aubin:YP}, and Schoen \cite{Schoen:YP}; we refer the reader to the survey article \cite{LeeParker}. 

In unpublished work, Hamilton introduced the Yamabe flow as a possible alternative method for solving the Yamabe problem and showed that the flow existed for all time. However, the problem of {convergence} turns out to be highly non-trivial.  For conformally flat metrics with positive Ricci curvature, Chow showed that the flow converged as $t\to\infty$ to a metric of constant scalar curvature \cite{Chow:YFlcf}. Ye removed the Ricci curvature condition \cite{Ye:YF} and, subsequently, Schwetlick and Struwe showed that the flow converged in dimensions $3\leq n \leq 5$ under the assumption that the starting Yamabe energy was ``not too large'' \cite{SS03}. The energy assumption was then removed by Brendle to establish unconditional convergence of the flow in dimensions $3\leq n\leq 5$ in \cite{Brendle:YF3to5}, and convergence in dimensions $n\geq 6$ under a technical hypothesis on the conformal class \cite{Brendle:HighDim}. 
We refer to \cite{Brendle:evoEqns} for a survey concerning these and related results.

Our work complements these contributions by showing that based on certain Morse-theoretic properties of the limit metric, the rate of convergence has either exponential or polynomial upper bounds, and in the latter case, the polynomial rate of convergence cannot in general be improved since it does in fact occur: therefore, it gives an essentially complete description of the rate of convergence for this flow. Perhaps the most novel outcome of this work is the result that there exist slowly converging geometric flows.

\subsection{Main Results} We now list our main results (we will define the precise terminology below), starting with the following statement concerning general upper bounds on the rate of convergence of the Yamabe flow. Integrability is defined in Definition \ref{IntegDef}.

\begin{theo}\label{theo:gen-rate-conv}
Assume that $g(t)$ is a Yamabe flow that is converging in $C^{2,\alpha}(M,g_{\infty})$ to $g_{\infty}$ as $t\to\infty$ for some $\alpha\in (0,1)$. Then, there is $\delta>0$ depending only on $g_{\infty}$ so that
\begin{enumerate}
\item If $g_{\infty}$ is an integrable critical point, then the convergence occurs at an exponential rate
\begin{equation*}
\Vert g(t) -  g_{\infty} \Vert_{C^{2,\alpha}(M,g_{\infty})} \leq C e^{-\delta t},
\end{equation*}
for some constant $C > 0$ depending on $g(0)$.
\item In general, the convergence cannot be worse than a polynomial rate
\begin{equation*}
 \Vert g(t) -  g_{\infty} \Vert_{C^{2,\alpha}(M,g_{\infty})} \leq C (1+t)^{-\delta},
\end{equation*}
for some constant $C > 0$ depending on $g(0)$.
\end{enumerate}
\end{theo}
The question of the rate of convergence of the flow was raised by Ye \cite[p.\ 36]{Ye:YF}. 
In general, the polynomial rate of convergence cannot be improved, as we discuss below.

Two previous results on this question are worth mentioning. 
First, Struwe's method of showing that the Yamabe flow on the $2$-sphere (which agrees with the Ricci flow in this case) converges exponentially fast  \cite{Struwe:curvFlows} can in fact be extended to prove that a Yamabe flow converging to the standard round metric on the sphere (in all dimensions) converges exponentially fast (this also follows from the work of Brendle \cite{Brendle:YFSn}).
We remark that this is a special case of case (1) of Theorem \ref{theo:gen-rate-conv} as the round metric is integrable by Obata's Theorem. Second, the convergence statement in case (2) of Theorem \ref{theo:gen-rate-conv} can in some sense be regarded as an implicit corollary of the arguments of \cite{Brendle:YF3to5} (because we are assuming the metric converges, there cannot be any bubbling phenomena; thus, one may use the remaining arguments in \cite{Brendle:YF3to5}, that may be verified to apply in any dimension, to conclude). The proof we give for Theorem \ref{theo:gen-rate-conv} is self-contained and applies, in a unified framework, to both settings. 
Moreover, the method we use directly applies to other gradient flows (e.g., the Calabi flow)
although we do not go into the details of such 
applications in this article.

The integrability condition is a nearly sharp condition for exponential convergence as is shown in the next theorem. 
The Adams--Simon positivity condition is defined in Definition  \ref{defi:ASp}.

\begin{theo}\label{theo:gen-slow-conv}
Assume that $g_{\infty}$ is a non-integrable critical point of the Yamabe energy with order of integrability $p\geq 3$. If $g_{\infty}$ satisfies the Adams--Simon positivity condition $AS_{p}$, then, there exists a metric $g(0)$
conformal to $g_{\infty}$ so that the Yamabe flow $g(t)$ starting from $g(0)$ exists for all time and converges in $C^{\infty}(M,g_{\infty})$ to $g_{\infty}$ as $t\to\infty$. The convergence occurs ``slowly'' in the sense that
\begin{equation*}
C^{-1}(1+t)^{- \frac{1}{p-2}} \leq \Vert g(t) -  g_{\infty} \Vert_{C^{2,\alpha}(M,g_{\infty})} \leq C(1+ t)^{- \frac{1}{p-2}},
\end{equation*}
for some constant $C>0$. 
\end{theo}

The bulk of this article is devoted to the proof of Theorem \ref{theo:gen-slow-conv}. Our proof is based on an adaptation of the remarkable tools developed by Simon and Adams--Simon in the study of isolated singularities of minimal surfaces and harmonic maps to the parabolic setting and, more specifically, to the Yamabe flow. As stated, three conditions need to be checked for a critical point $g_{\infty}$ to be a limit point of a slowly converging Yamabe flow. Integrability and degeneracy are defined in Definition \ref{IntegDef} and Lemma \ref{lemm:int-imp-F-zero}. The degeneracy can be studied by looking at the spectrum of the Laplace operator $\Delta_{g_{\infty}}$. For a degenerate metric, determining integrability (or lack thereof) depends on understanding the set of constant scalar curvature metrics near $g_\infty$. For instance, if $g_{\infty}$ is {isolated} and degenerate it must be non-integrable. Lastly, the Adams--Simon condition $AS_{p}$ (Definition \ref{defi:ASp}) concerns the first non-trivial term in the analytic expansion of the Lyapunov--Schmidt reduction $F$ of the Yamabe functional at $g_{\infty}$, whose order we denote by $p$. 

We give two criteria, of different nature, to check the condition $AS_p$:
\begin{itemize}
\item{The condition $AS_{3}$ is satisfied whenever
\begin{equation*}
F_{3}(v) =-2(N-1)(N-2) R_{g_{\infty}} \int_{M} v^{3} \ dV_{g_{\infty}}
\end{equation*}
does not identically vanish on the nullspace $\Lambda_{0}$, which is the linear span of functions $v$ such that $(n-1)\Delta_{g_{\infty}}v+R_{g_{\infty}}v=0$. 
Of course, this can in principle be computed once $\Lambda_{0}$ is explicitly known.}
\item{When $p>3$, the condition $AS_{p}$ holds if $g_{\infty}$ is both degenerate and a strict local minimum of the Yamabe functional.}
\end{itemize}

The relevance of the second criterion is related to the solution of the Yamabe problem: if $(M,\left[g\right])$ is not the round sphere with the associated conformal structure, then the Yamabe functional, $\cY$, is {coercive} and thus has a global minimum $g_{\min}$. As such, the existence of polynomially converging flows is guaranteed by Theorem \ref{theo:gen-slow-conv} whenever $g_{\min}$ is isolated but degenerate (which is simply a condition on the spectrum of the Laplacian of $g_{\min}$).  

We give examples of applicability of these criteria in the following two propositions. 
\begin{prop}\label{prop:CPn-example}
Fix integers $n,m > 1$ and a closed $m$-dimensional Riemannian manifold $(M^{m},g_{M})$ with constant scalar curvature $R_{g_{M}} \equiv 4(n+1)(m+n-1)$. We denote the complex projective space equipped with the Fubini--Study metric by $(\CC P^{n},g_{FS})$, where the normalization of $g_{FS}$ is fixed so that $\SS^{2n+1}(1)\to(\CC P^{n},g_{FS})$ is a Riemannian submersion. Then, the product metric $(M^{m}\times \CC P^{n}, g_{M} \oplus g_{FS})$ is a degenerate critical point satisfying $AS_{3}$. 
\end{prop}

We remark that any closed manifold $(M,g_{M})$ whose scalar curvature is a positive constant may be rescaled so as to satisfy the conditions of the previous proposition. 

\begin{prop}\label{prop:example}
Let $n>2$.
The product metric on $\SS^{1}\left(\frac 1 {\sqrt{n-2}}\right) \times \SS^{n-1}(1)$ is a non-integrable critical point satisfying $AS_{p}$ for some $p\geq 4$.  
\end{prop}

There are not many examples of degenerate critical points of geometric functionals where non-integrability can be checked, cf.\ \cite[\S 5]{AdamsSimon}. In fact, it seems that our second example is the first of a critical point which satisfies $AS_{p}$ for $p >3$ (cf.\ \cite[Remark 1.19]{AdamsSimon} where the authors explain a method for checking $AS_{3}$ that does not work for $p>3$).

In conclusion, we may construct examples of slowly converging Yamabe flows in a range of conformal classes
and in any dimension greater than 2.
For example, Proposition \ref{prop:CPn-example} yields examples which are not conformally flat, while the metrics in Proposition \ref{prop:example} are locally conformally flat. 

\begin{coro}
There exists a Yamabe flow in the conformal class of the metrics described in Propositions \ref{prop:CPn-example} and \ref{prop:example} that converges to the given metrics exactly at a polynomial rate, as in Theorem~\ref{theo:gen-slow-conv}.
\end{coro}

This seems to be the first construction of a slowly converging flow in the setting of geometric flows
of parabolic type. We expect that our methods can be adapted (possibly with the added difficulty of a large gauge group) to produce slowly converging flows for other (possibly degenerate) parabolic flows.

\subsection{Outline of proof of Theorem \ref{theo:gen-slow-conv}}
The proof of Theorem \ref{theo:gen-slow-conv} appears in Section 4. 
It is rather long and, at times, quite technical,
and so we take the opportunity here to outline its structure.

Ultimately, we would like to construct a solution $u(t)$
to a quasilinear parabolic equation that converges 
at a rate $O((1+t)^{-\frac1{p-2}})$ to the constant function $1$.
Intuitively, one expects the slow convergence to be 
``generated" infinitesimally by the non-integrable
directions, namely from $\Lambda_0$.
At the same time, in order to use a fixed-point argument
to generate a polynomially converging flow it is most convenient
to have a guess as to what the leading order behavior of the
flow ought to be, and then show that the actual solution
is a small perturbation (of order $o((1+t)^{-\frac1{p-2}}$)
of this guess.
Fortunately, the Adams--Simon condition precisely furnishes
such an ansatz: the function $\varphi(t)$ 
of Lemma \ref{lemm:poly-decay-def-varphi}. The proof of Theorem \ref{theo:gen-slow-conv} thus 
amounts to showing that we can find $u(t)$ solving the Yamabe flow with
\begin{equation}\label{eq:outline-u-desired}
u(t)-\varphi(t)=o((1+t)^{-\frac1{p-2}}.
\end{equation}
We now explain the different steps to derive this estimate.

\noindent
{\it Step 1.}
Firstly, we would like to understand separately 
the behavior of the flow in $\Lambda_0$ (the kernel of the linearized Yamabe operator) and
$\Lambda_0^\perp$ (the orthogonal complement) directions.
Thus, in Proposition \ref{prop:slow-decay-rewrite-two-comp} we prove that the Yamabe flow
is equivalent to two flows: the {\it kernel-projected
flow} that takes place on $\Lambda_0$, and the
{\it kernel orthogonal-projected flow} that takes place on $\Lambda_0^\perp$.
There are two key points about this result that
make it useful.
First, the kernel orthogonal-projected flow (see \eqref{kopfEq}) is 
a perturbation of a {\it linear} parabolic equation.
In other words, \eqref{kopfEq} is, of course, nonlinear,
but it can be considered as a linear equation
since we prove an a priori estimate on the error term.
Second, the kernel-projected flow (see \eqref{kpfEq}) is
a perturbation of a system of ODEs. Again, we have
a  precise a priori estimate on the error term. That these estimates are sufficiently strong will play a crucial role in a contraction mapping argument discussed in Step 4 below. 

The proof of Proposition \ref{prop:slow-decay-rewrite-two-comp} involves some rather tedious
computations. First, in Lemma \ref{lemm:DYu-est}, we reduce the Yamabe flow
to the situation of a gradient flow. Indeed, the Yamabe
flow is certainly a gradient flow on the level of metrics
but that is not quite the case of the level of conformal factor. After this preparatory step, we project
the flow onto $\Lambda_0$ and its orthogonal complement,
and try to reduce the resulting equations to \eqref{kopfEq} and \eqref{kpfEq}.
This is essentially a consequence of 
 multiple applications of Taylor's theorem, the
Lyapunov--Schmidt reduction (Proposition \ref{prop:LSred}), and the estimates
on $D^3\cY$ (Appendix \ref{sec:app}). It also uses the fact that
the ansatz $\varphi(t)$ allows for an important cancellation
that considerably simplifies the kernel-projected piece
and thus allows to reduce it to a system of ODEs.

\noindent
{\it Step 2.}
We obtain a solution to the kernel-projected flow in a weighted H\"older norm in Lemma \ref{lemm:ker-proj-ODE}. This norm precisely
captures a polynomial rate of decay of this solution.

\noindent
{\it Step 3.}
We obtain a solution to the kernel orthogonal-projected flow,
again in a weighted norm, in Lemma \ref{lemm:ker-orthog-prob}. Here some care is needed
since we must again work with parabolic H\"older norms.

\noindent
{\it Step 4.}
Finally, in Proposition \ref{prop:contract-map},
we set up a fixed point argument in a Banach space
that uses the two different weighted norms (on
$\Lambda_0$ and $\Lambda_0^\perp$). Here one needs
to be quite careful with the order of decay 
of the error terms collected in the previous three
steps, in order to show that the map is a contraction. Once we have shown that the map is a contraction, we have existence of a Yamabe flow. Moreover, the flow satisfies the estimate \eqref{eq:outline-u-desired} because we define the weighted norms so that any function which is a perturbation of $\varphi(t)$ in the given norm will fall off at a rate faster than $\varphi(t)$. 

\subsection{Structure of the article}
Section \ref{sec:def-prelim} is devoted to fixing the notation and recalling some basic facts about the (volume constrained) Yamabe functional, its analyticity and the Lyapunov--Schmidt reduction near a critical point. In Section \ref{sec:loj-ineq}, we use the \L ojasiewicz--Simon inequality to prove Theorem \ref{theo:gen-rate-conv}. Then, in Section \ref{sec:slow-flow} we study polynomial convergence phenomena for non-integrable critical points and in Section \ref{sec:examp} we prove Propositions \ref{prop:CPn-example} and \ref{prop:example}.
 The computation of the third variation of the reduced Yamabe energy (namely of the formula for $F_{3}$ given above) is contained in Appendix \ref{sec:app}. Appendix \ref{app:mon-period} contains a proof of a technical lemma needed for the proof of Proposition \ref{prop:example}

\section{Definitions and Preliminaries}\label{sec:def-prelim}

The \emph{Yamabe functional} is defined by
\begin{equation*}
\cY(g):= \Vol(M,g)^{-\frac 2 N}  \int_{M}R_{g}\, dV_{g},
\end{equation*}
where $dV_g$ is the Riemannian volume form associated to $g$, $R_g$ denotes the scalar curvature of $g$ and 
\begin{equation*}
N =\frac{2n}{n-2}.
\end{equation*}
If $g = w^{N-2}g_{\infty}$ for some positive $w\in C^{2}(M)$ and smooth metric $g_{\infty}$, then an alternative expression
for the Yamabe functional (restricted to the conformal class of $g$)
is 
\begin{equation*}
\cY(w) = \frac{\int_{M} \left( (N+2) |\nabla_{g_{\infty}}w|^{2} + R_{g_{\infty}}w^{2}\right)dV_{g_{\infty}} }{\left( \int_{M} w^{N} dV_{g_{\infty}}\right)^{\frac 2N }},
\end{equation*}
since $R_{w^{N-2}g_{\infty}}=w^{1-N}(R_{g_{\infty}}w-(N+2)\Delta_{g_{\infty}}w)$.

Consider the \emph{unit volume conformal class} associated to the metric $g_{\infty}$
\begin{equation*}
[g_{\infty}]_{1} : = \left\{ w^{N-2} g_{\infty} : w \in C^{2,\alpha}(M), w> 0, \int_{M} w^{N} dV_{g_{\infty}} = 1\right\}.
\end{equation*}
In order to avoid ambiguities, we define the following notation: for $k\in \NN$, we denote the $k$-th differential of the Yamabe functional on $[g_{\infty}]_{1}$ at the point $w$ in the directions $v_{1},\dots, v_{k}$ by
\begin{equation*}
D^{k}\cY(w)[v_{1},\dots,v_{k}].
\end{equation*}
As we will see from \eqref{eq:L2C0}, the functional $v\mapsto D^{k}\cY(w)[v_{1},\dots,v_{k-1},v]$ is in the image of $L^{2}(M,g_{\infty})$ under the natural embedding into $C^{2,\alpha}(M,g_{\infty})'$. Therefore, we will also write
\begin{equation*}
D^{k}\cY(w)[v_{1},\dots,v_{k-1}]
\end{equation*}
for this element of $L^{2}(M,g_{\infty})$. When $k=1$, we will drop the (second) brackets, and thus consider $D\cY(w) \in L^{2}(M,g_{\infty})$. 

We may write the differential of $\cY$ restricted to $[g_{\infty}]_{1}$ as
\begin{align*}
\frac 1 2 D\cY(w)[v]& =  
 \int_{M}\left[- (N+2) \Delta_{g_{\infty}} w + R_{g_{\infty}} w   - r_{w^{N-2}g_{\infty}} w^{N-1}\right] v\,dV_{g_{\infty}}\\
& = \int_{M}(R_{w^{N-2}g_{\infty}} - r_{w^{N-2}g_{\infty}})w^{N-1}v \,dV_{g_{\infty}},
\end{align*}
for $v \in C^{2,\alpha}(M,g_{\infty})$. Here,
\begin{equation*}
r_{g} = \Vol(M,g)^{-1}\int_{M} R_{g} dV_{g}.
\end{equation*}
Regarded as an element of $L^{2}(M,g_{\infty})$, we have that
\begin{equation}\label{eq:Dy-L2}
\frac 12 D\cY(w) = -(N+2) \Delta_{\infty} w + R_{g_{\infty}}w - r_{w^{N-2}g_{\infty}}w^{N-1}.
\end{equation}
As above, we have associated the metric $w^{N-2}g_{\infty}$ to the function $w$. This is clearly a bijection, so we will continue to do so throughout. Thus, a unit volume metric $g_{\infty}$ is a critical point for the Yamabe energy $\cY$ restricted to $[g_{\infty}]_{1}$ exactly when $g_{\infty}$ has constant scalar curvature. 

We now fix $g_{\infty}$ to be a unit volume, constant scalar curvature metric. We denote by $CSC_{1}$ the set of unit volume, constant scalar curvature metrics in $[g_{\infty}]_{1}$ and further define the \emph{linearized Yamabe operator at $g_{\infty}$}, $\cL_{\infty}$, by means of the formula
\begin{equation*}
-(N-2)\int_{M}w \cL_{\infty} v \,dV_{g_{\infty}}: = \frac 1 2D^{2}\cY(g_{\infty})[v,w] = \frac 12 \frac{d}{ds}\Big|_{s=0}D \cY((1+sv)^{N-2}g_{\infty})[w],
\end{equation*}
for $v \in C^{2}(M)$. A  computation shows that 
\begin{equation*}
\cL_{\infty} v = (n-1) \Delta_{g_{\infty}} v + R_{g_{\infty}}v. 
\end{equation*}
We define $\Lambda_{0} := \ker \cL_{\infty} \subset L^{2}(M,g_{\infty})$. 

Spectral theory shows that $\Lambda_{0}$ is finite dimensional (it is the eigenspace of the Laplacian for the eigenvalue $\frac{R_{g_{\infty}}}{n-1}$). We will write $\Lambda_{0}^{\perp}$ for the $L^{2}(M,g_{\infty})$-orthogonal complement. It is crucial throughout this work that the Yamabe functional is an analytic map. Here, we will mean analytic in the sense of \cite[Definition 8.8]{Zeidler:NFA1}.
\begin{lemm}\label{AnalyticLemma}
Fix a metric $g_{\infty}$. 
The Yamabe functional is an analytic functional on $\{u\in C^{2,\alpha}(M,g_{\infty}):u>0\}$ in the sense that for each $w_{0} \in C^{2,\alpha}(M,g_{\infty})$ with $w_{0}>0$, there is $\epsilon > 0$ and bounded multi-linear operators for each $k\geq 0$
\begin{equation*}
\cY^{(k)} : C^{2,\alpha}(M,g_{\infty})^{\times k}\to \RR,
\end{equation*}
so that if $\Vert w - w_{0}\Vert_{C^{2,\alpha}} < \epsilon$, then
$\sum_{k=0}^{\infty} \Vert \cY^{(k)}\Vert \cdot \Vert w-w_{0}\Vert_{C^{2,\alpha}}^{k} < \infty,
$
and
\begin{equation*}
\cY(w) = \sum_{k=0}^{\infty} \cY^{(k)}\underbrace{(w-w_{0},w-w_{0},\dots,w-w_{0})}_{k\ \textrm{times}} \ \textrm{in} \ C^{2,\alpha}(M,g_{\infty}).
\end{equation*}
\end{lemm}
It is not hard to verify this, by simply expanding the denominator of $\cY$ in a power series around $\left(\int_{M} w_{0}^{N}dV_{g_{\infty}}\right)^{-\frac N2}$ and noting that the numerator is already a bilinear function in $w$. Now, by a standard Lyapunov--Schmidt reduction \cite[Theorem 4.H]{Zeidler:NFA1}, \cite[\S 3]{Simon:RegularityEnergyMinMaps}), we may use the Implicit Function Theorem to show the following.
\begin{prop}\label{prop:LSred}
There is $\epsilon > 0$ and an analytic map $\Phi: \Lambda_{0} \cap \{ v: \Vert v \Vert_{L^{2}} < \epsilon \} \to C^{2,\alpha}(M,g_{\infty}) \cap \Lambda_{0}^{\perp}$ so that $\Phi(0) =0$, $D\Phi(0) = 0$, 
\begin{equation}\label{eq:LS-est-Dphi}
\sup_{\substack{\Vert v \Vert_{L^{2}} <\epsilon\\  \Vert w \Vert_{L^{2}} \leq 1}}\Vert D\Phi(v)[w]\Vert_{L^{2}} < 1,
\end{equation}
and so that defining $\Psi (v) = 1+ v+ \Phi(v)$, we have that $\Psi(v) > 0$, $\Vol(M,\Psi(v)^{N-2}g_{\infty}) = 1$ and
\begin{equation*}
\proj_{\Lambda_{0}^{\perp}}[D{\cY}(\Psi(v))]=\proj_{\Lambda_{0}^{\perp}} \left[ \left(R_{\Psi(v)^{N-2}g_{\infty} }  - r_{\Psi(v)^{N-2}g_{\infty} } \right) \Psi(v)^{N-1} \right] = 0.
\end{equation*}
Furthermore,
\begin{equation*}
\proj_{\Lambda_{0}}[D{\cY}(\Psi(v))]= \proj_{\Lambda_{0}} \left[ \left(R_{\Psi(v)^{N-2}g_{\infty} }  - r_{\Psi(v)^{N-2}g_{\infty} } \right) \Psi(v)^{N-1} \right] = D F,
\end{equation*}
where $F: \Lambda_{0} \cap \{ v: \Vert v \Vert_{L^{2}} \leq \epsilon \} \to\RR$ is defined by $F(v) = \cY(\Psi(v))$. 
Finally, the intersection of $CSC_{1}$ with a small $C^{2,\alpha}(M,g_{\infty})$-neighborhood of $1$ coincides with
\begin{equation*}
\cS_{0}:= \left\{ \Psi(v) : v \in \Lambda_{0},\Vert v \Vert_{L^{2}} < \epsilon, D F (v) = 0\right\},
\end{equation*}
which is a real analytic subvariety (possibly singular) of the following $(\dim\Lambda_{0})$-dimensional real analytic submanifold of $C^{2,\alpha}(M,g_{\infty})$:
\begin{equation*}
\cS := \left\{ \Psi(v) : v \in \Lambda_{0}, \Vert v \Vert_{L^{2}} < \epsilon\right\}.
\end{equation*}
\end{prop}
This follows in the usual way from the analytic implicit function theorem, cf., \cite[Corollary 4.23]{Zeidler:NFA1}. We will refer to $\cS$ as the \emph{natural constraint} for the problem.

\begin{defi}
\label{IntegDef}
For $g_{\infty} \in CSC_{1}$, we say that $g_{\infty}$ is \emph{integrable} if for all $v \in \Lambda_{0}$, there is a path $w(t) \in C^{2}((-\epsilon,\epsilon) \times M,g_{\infty})$ so that $w(t)^{N-2}g_{\infty}\in CSC_{1}$ and $w(0) =1$, $w'(0) = v$. Equivalently, $g_{\infty}$ is integrable if and only if $CSC_{1}$ agrees with $\cS$ in a small neighborhood of $1$ in $C^{2,\alpha}(M,g_{\infty})$. 
\end{defi}

For our purposes, the following equivalent characterization of integrability is crucial:

\begin{lemm}[{\cite[Lemma 1]{AdamsSimon}}]\label{lemm:int-imp-F-zero}
Integrability as defined above is equivalent to the functional $F$ (as defined in Proposition \ref{prop:LSred}, the Lyapunov--Schmidt reduction) being constant on a neighborhood of $0$ inside $\Lambda_{0}$. 
\end{lemm}

We remark that if $\Lambda_{0} = 0$, i.e., if $\cL_{\infty}$ is injective, it is standard to call $g_{\infty}$ a \emph{non-degenerate} critical point; if this holds, $g_{\infty}$ is automatically integrable in the above sense. 
On the other hand, if $\Lambda_{0}$ is non-empty, then we call $g_{\infty}$ \emph{degenerate}; we emphasize that there are many examples of degenerate metrics, see e.g.,\ \cite{BettiolPiccione:mult-YP-subm}.

Now, suppose that $g_{\infty}$ is a non-integrable critical point. Because $F(v)$ is analytic (it is the composition of two analytic functions), we may expand it in a power series
\begin{equation*}
F(v) = F(0) + \sum_{j\geq p}F_{j}(v),
\end{equation*}
where $F_{j}$ is a degree $j$ homogeneous polynomial on $\Lambda_{0}$ and $p$ is chosen so that $F_{p}$ is nonzero. We will call $p$ the \emph{order of integrability} of $g_{\infty}$. We will also need a further hypothesis for non-integrable critical points, introduced in \cite{AdamsSimon}.

\begin{defi}\label{defi:ASp}
We say that $g_{\infty}$ satisfies the \emph{Adams--Simon positivity condition}, $AS_{p}$ for short (here $p$ is the order of integrability of $g_{\infty}$), if it is non-integrable and $F_{p}|_{\SS^{k}}$ attains a positive maximum for some $\hat v \in \SS^{k} \subset \Lambda_{0}$. Recall that $F_{p}$ is the lowest degree non-constant term in the power series expansion of $F(v)$ around $0$ and $\SS^{k}$ is the unit sphere\footnote{Here we are using the inner product induced on $\Lambda_{0}$ coming from the $L^{2}$ inner product on $T_{1}[g_{\infty}]_{1}$.} in $\Lambda_{0}$. 
\end{defi}

The Adams--Simon positivity condition is ultimately needed for the construction of the function $\varphi$ in Lemma \ref{lemm:poly-decay-def-varphi}, that serves as an approximate solution to Yamabe flow converging at a polynomial rate. It is an interesting question whether or not $AS_{p}$ is a necessary condition for the existence of slowly converging examples in the elliptic and parabolic settings.

An important observation is that when the order of integrability, $p$, is odd, the Adams--Simon positivity condition is always satisfied. Moreover, the order of integrability (at a critical point of $\cY$) always satisfies $p\geq 3$ as we recall in Appendix A. 
Furthermore, we show there that
\begin{equation}\label{eq:F3-comp}
F_{3}(v) =-2(N-1)(N-2) R_{g_{\infty}} \int_{M} v^{3} \ dV_{g_{\infty}}.
\end{equation}

\section{The \L ojasiewicz--Simon inequality and rate of convergence}\label{sec:loj-ineq}

One of the tools for controlling the rate of convergence of the Yamabe flow will be the \L ojasiewicz--Simon inequality. This was first proven for a certain class of geometric functionals by  Simon \cite{Simon:AnnalsEvolution}, who showed that the classical \L ojasiewicz inequality for analytic functions in finite dimensions could be extended to a Banach space setting. 

\begin{defi}[\L ojasiewicz--Simon inequality]
Suppose that $\mathscr{B}$ is a Banach space and $U \subset \mathscr{B}$ is an open subset.  Fix a functional $E\in C^{2}(U,\RR)$ and denote by $DE\in C^{1}(U,\mathscr{B}')$  its first derivative (here $\mathscr{B}'$ is the dual Banach space to $\mathscr{B}$). We will additionally fix a Banach space $\mathscr{W}$ with a continuous embedding $\mathscr{W} \hookrightarrow \mathscr{B}'$. For $x_{0}\in U$ a critical point of $E$, i.e., $DE(x_{0}) = 0$, we say that $E$ satisfies the \emph{\L ojasiewicz--Simon inequality} with exponent $\theta \in (0,\frac 12]$ near $x_{0}$ if there exists a neighborhood $x_{0}\in V\subset U$ as well as constants $C>0$ so that 
\begin{equation*}
|E(x) - E(x_{0})|^{1-\theta}\leq C \Vert DE(x)\Vert_{\mathscr{W}},
\qquad\hbox{for all $x \in V$}.
\end{equation*}
 
\end{defi}

Notice that if $\mathscr{B} = \mathscr{W} = \RR^{n}$, this reduces to the classical \L ojasiewicz inequality \cite{Loj:ineq}. The \L ojasiewicz--Simon inequality has recently received much attention; we will apply the following general result to show that it holds in our setting:
\begin{prop}[{\cite[Theorem 3.10]{Chill:LS-ineq}}]\label{prop:chill}
Fix $\mathscr{B}$, $U\subset \mathscr{B}$, $E\in C^{2}(U,\RR)$, $\mathscr{W} \hookrightarrow \mathscr{B}'$ and $x_{0} \in U$ with $DE(x_{0}) =0$ as in the previous definition. We also define the second derivative $\mathscr{L} := D^{2}E \in C(U,\cB(\mathscr{B},\mathscr{B}'))$, where $\cB(\mathscr{B},\mathscr{B}')$ is the space of continuous maps between the Banach spaces $\mathscr{B}$ and $\mathscr{B}'$. We will suppose that the following hypotheses are satisfied:
\begin{enumerate}
\item[(A)] The kernel $\ker\mathscr{L}(x_{0}) \subset \mathscr{B}$ is complemented in $\mathscr{B}$, i.e., there exists a projection $P \in \cB(\mathscr{B},\mathscr{B})$ so that $\range P = \ker\mathscr{L}(x_{0})$. It follows from this that $\mathscr{B} = \ker \mathscr{L}(x_{0}) \oplus \ker P$ is a topological direct sum. Denote by $P'\in\cB(\mathscr{B}',\mathscr{B}')$ the 
adjoint  map.
\item[(B1)]  The map $
\mathscr{W}\hookrightarrow 
\mathscr{B}'
$ is a continuous embedding.
\item[(B2)] The adjoint projection $P'$ leaves 
$\mathscr{W}$ 
invariant.
\item[(B3)] The map $DE\in C^{1}(U,
\mathscr{W})
$.
\item[(B4)] We have $\range 
\mathscr{L}(x_{0}) 
= \ker P' \cap 
\mathscr{W}
$. 
\end{enumerate}
Under these hypothesis, we may find a neighborhood $U_{0}$ of $0$
in $\ker \mathscr{L}(x_{0})
$ and a neighborhood $U_{1}$ of $0$ in $\ker P$ as well as a function $H \in C^{1}(U_{0},U_{1})$ parametrizing the natural constraint, i.e.,
\begin{equation*}
\{ x \in U_{0} + U_{1} : DE(x_{0}+x) \in (\ker \sL(x_{0}))'\} = \{ x + H(x) : x \in U_{0}\}.
\end{equation*}
Recall that the natural constraint is then 
\begin{equation*}
S:=\{ x_{0} + x + H(x) : x \in U_{0}\}. 
\end{equation*}
Finally, suppose that
\begin{enumerate}
\item[(C)] The function $E(x_{0}+\cdot)$ satisfies the \L ojasiewicz inequality on the natural constraint $S$ with exponent $\theta \in (0,\frac 12]$. More precisely, we assume that 
\begin{equation*}
|E(x_{0}+x+H(x)) - E(x_{0})|^{1-\theta}\leq C \Vert DE(x_{0}+x+H(x))\Vert_{\mathscr{W}}, \qquad \hbox{for all $ x\in U_{0}$}. 
\end{equation*}
\end{enumerate}
Then the functional $E$ satisfies the \L ojasiewicz--Simon inequality near $x_{0}$ with the same exponent $\theta \in (0,\frac 12]$. 
\end{prop}

\begin{prop}\label{prop:LS-ineq}
Suppose that $g_{\infty}$ is a unit volume constant scalar curvature metric. There are $\theta \in (0,\frac 12 ]$, $\epsilon > 0$ and $C>0$ only depending on $n$ and $g_{\infty}$ so that for $u \in C^{2,\alpha}(M,g_{\infty})$ with $\left\|u-1\right\|_{C^{2,\alpha}(M,g_{\infty})} < \epsilon$ and $\Vol(M,u^{N-2}g_{\infty}) = 1$, then
\begin{equation*}
|r_{u^{N-2}g_{\infty}} - r_{g_{\infty}}|^{1-\theta} \leq C \Vert D \cY(u^{N-2}g_{\infty})\Vert_{L^{2}(M,g_{\infty})}.
\end{equation*}
If $g_{\infty}$ is an integrable critical point then $\theta = \frac 12$. If $g_{\infty}$ is non-integrable, then $\theta = \frac{1}{p}$ where $p$ is the order of integrability of $g_{\infty}$. 
\end{prop}
\begin{proof}
To verify this, we will show that the hypothesis of Proposition \ref{prop:chill} are satisfied for the Yamabe energy $\cY$. 
We work with the Banach spaces $\mathscr{B}:=
C^{2,\alpha}(M,g_{\infty})$ 
and $\mathscr{W} : = L^{2}(M,g_{\infty})$,
and fix $U$ a small enough ball around $1$ in $C^{2,\alpha}(M,g_{\infty})$ so that Proposition \ref{prop:LSred} is applicable in $U$. 

Hypothesis (A) is the statement that $\Lambda_{0} = \ker\cL_{\infty}$ is complemented in $C^{2,\alpha}(M,g_{\infty})$, which is immediate by the following argument. One first checks that the $L^{2}$-projection map $\proj_{\Lambda_{0}}$ restricts to a continuous map from $C^{2,\alpha}(M,g_{\infty})$ onto $\Lambda_{0}$ (since of course $C^{2,\alpha}(M,g_{\infty})\hookrightarrow L^{2}(M,g_{\infty})$ as a continuous embedding); from this, it follows (cf.\ \cite[p.\ 580]{Chill:LS-ineq}) that $\Lambda_{0}'$ is complemented (by the map $\proj_{\Lambda_{0}}'$) in the dual space $C^{2,\alpha}(M,g_{\infty})'$, and its complement $\Lambda_{0}'^{\perp}$ may be canonically identified with $(\Lambda_{0}^{\perp})'$.

Hypothesis (B) is satisfied as follows: consider the map
\begin{equation}\label{eq:L2C0}
\begin{split}
\mathscr{W} : =
 L^{2}(M,g_{\infty}) & \hookrightarrow C^{2,\alpha}(M,g_{\infty})'\\
f & \mapsto \left(\varphi\mapsto \int_{M} f\varphi \, dV_{g_{\infty}} \right).
\end{split}
\end{equation}
\begin{enumerate}[(B1)]
\item This map is continuous. 
\item The map $\proj_{\Lambda_{0}}'\in \cB(C^{2,\alpha}(M,g_{\infty})')$ leaves $L^{2}(M,g_{\infty})$ invariant (of course, here we are considering the composition $\proj_{\Lambda_{0}}:C^{2,\alpha}(M,g_{\infty})\to \Lambda_{0}\hookrightarrow C^{2,\alpha}(M,g_{\infty})$).
\item That $D \cY \in C^{1}(U,L^{2}(M,g_{\infty}))$ follows from the explicit form of $D \cY$ given above. 
\item Finally, we must verify that $\range\cL_{\infty} = \Lambda_{0}'^{\perp}\cap L^{2}(M,g)$. That $\range\cL_{\infty} \subseteq \Lambda_{0}'^{\perp}\cap L^{2}(M,g)$ is obvious because $\cL_{\infty}$ is formally self-adjoint on $L^{2}$. The other inclusion follows from the $L^{2}$-spectral decomposition of $\cL_{\infty}$. 
\end{enumerate}

Thus to prove the \L ojasiewicz--Simon inequality with exponent $\theta \in (0,\frac 12]$, it is enough to check (C), i.e., that the Yamabe energy restricted to the natural constraint satisfies the \L ojasiewicz--Simon inequality with exponent $\theta \in (0,\frac 12]$. Recall that in Proposion \ref{prop:LSred} we have defined $F(v) = \cY(\Psi(v))$. In the integrable case, clearly $F(v) \equiv F(0)$, so $F$ satisfies the  \L ojasiewicz--Simon inequality for $\theta = \frac 12$.

In general, $F$ is an analytic function whose power series has its first nonzero degree $p$, by definition. Thus, we may conclude (cf.\ \cite[Proposition 2.3(b)]{Chill:LS-ineq}) that $F$ satisfies the  \L ojasiewicz--Simon inequality with exponent $\theta =\frac{1}{p}$. 

The claim follows from this---we have replaced $\cY(u)$ with $r_{u^{N-2}g_{\infty}}$ in the left hand side of the inequality by using the assumption that $u^{N-2}g_{\infty}$ has unit volume. 
\end{proof}

Now, we show how the \L ojasiewicz--Simon inequality yields quantitative estimates on the rate of convergence of the Yamabe flow. 

\begin{proof}[Proof of Theorem \ref{theo:gen-rate-conv}]
We consider a Yamabe flow $g(t)=u(t)^{N-2}g_{\infty}$ which converges to $g_{\infty}$ in $C^{2}(M,g_{\infty})$ as $t\to\infty$. We may assume without loss of generality that $g_{\infty}$ and thus $g(t)$ have unit volume. In Proposition \ref{prop:LS-ineq} we have shown that there is a \L ojasiewicz--Simon inequality near $g_{\infty}$ for some $\theta \in (0,\frac 12]$. 
We emphasize that if we are regarding $D\cY(g(t))$ as an element of $L^{2}(M,g_{\infty})$, i.e.,
$D\cY(g(t))=2\left(R_{g(t)}-r_{g(t)}\right)u(t)^{N-1}$.

Choose $t_{0}$ so that for $t \geq t_{0}$, $\Vert u(t) - 1\Vert_{C^{0}(M,g_{\infty})} \leq \frac 12$. We then have that
\begin{align*}
\frac{d}{dt} \left( r_{g(t)} - r_{g_{\infty}}\right) &  = -\frac {2}{N-2}
 \int_{M} \left( R_{g(t)} - r_{g(t)}\right)^{2} u(t)^{N} dV_{g_{\infty}}\\
& \leq 
- c \int_{M} \left( R_{g(t)} - r_{g(t)}\right)^{2} u(t)^{2N-2} dV_{g_{\infty}}\\
& = - c \Vert D \cY(g(t))\Vert_{L^{2}(M,g_{\infty})}^{2}\\
& \leq -c (r_{g(t)} - r_{g_{\infty}})^{2-2\theta}.
\end{align*}
where $c >0$ is a constant depending only on $n$ and $g_{\infty}$ (that we let change from line to line). Let us first assume that the \L ojasiewicz--Simon inequality is satisfied with $\theta = \frac 12$, i.e., that we are in the integrable case. The previous inequality yields
$r_{g(t)} - r_{g_{\infty}} \leq C e^{-2\delta t},
$
for $\delta>0$ depending only on $n$ and $g_{\infty}$ and $C>0$ depending on $g(0)$ (chosen so that this actually holds for all $t\geq 0$). On the other hand, if the \L ojasiewicz--Simon inequality holds with $\theta \in (0,\frac 12)$ then the same argument shows that 
$r_{g(t)} - r_{g_{\infty}} \leq C(1+t)^{\frac{1}{2\theta -1}}.
$

Recall 
that the evolution equation for the conformal factor $u=u(t)$ is given by 
\begin{equation*}
\frac{\partial u}{\partial t}=-\frac{u}{N-2}\left(R_{g(t)}-r_{g(t)}\right).
\end{equation*}
Thus, exploiting the fact that the flow converges in $C^2$ we may use the \L ojasiewicz--Simon inequality to compute
\begin{align*}
\frac{d}{dt} \left( r_{g(t)}-r_{g_{\infty}}\right)^{\theta} & = \theta \left( r_{g(t)}-r_{g_{\infty}}\right)^{\theta-1} \frac{d}{dt} \left( r_{g(t)}-r_{g_{\infty}}\right)\\
& \leq - c\theta  \left( r_{g(t)}-r_{g_{\infty}}\right)^{\theta-1} \Vert D \cY(g(t))\Vert_{L^{2}(M,g_{\infty})}^{2}\\
& \leq -c\theta \Vert D \cY(g(t))\Vert_{L^{2}(M,g_{\infty})}\\
& \leq -c\theta \left \Vert \frac {\partial u}{\partial t} \right \Vert_{L^{2}(M,g_{\infty})}.
\end{align*}
Thus, if $\theta = \frac 12$ 
(recall $\lim_{t\to\infty}u(t)=1$),
\begin{align*}
\Vert u(t) -1 \Vert_{L^{2}(M,g_{\infty})} & \leq \int_{t}^{\infty}  \left \Vert \frac {\partial u}{\partial s} \right \Vert_{L^{2}(M,g_{\infty})} ds\\
& \leq - c\int_{t}^{\infty} \frac{d}{ds}
\left[ \left( r_{g(s)}-r_{g_{\infty}}\right)^{\frac 12}\right]ds\\
& = c \left( r_{g(t)}-r_{g_{\infty}}\right)^{\frac 12}
\leq C e^{-\delta t}.
\end{align*}
A similar computation if $\theta \in (0,\frac 12)$ yields
$
\Vert u(t) -1 \Vert_{L^{2}(M,g_{\infty})}  \leq C (1+t)^{-\frac{\theta}{1-2\theta }}.
$

To obtain $C^{2}$ estimates, we may interpolate between $L^{2}(M,g)$ and $W^{k,2}(M,g)$ for $k$ large enough: interpolation \cite[Theorem 6.4.5]{BerghLofstrom:interpolationSpaces} and Sobolev embedding yields some constant $\eta \in (0,1)$ so that
\begin{equation*}
\Vert u(t) -1 \Vert_{C^{2,\alpha}(M,g_{\infty})} \leq \Vert u(t) -1 \Vert_{L^{2}(M,g_{\infty})}^{\eta} \Vert u(t) -1 \Vert_{W^{k,2}(M,g_{\infty})}^{1-\eta}.
\end{equation*}
Because $u(t)$ is converging to $1$ in $C^{2,\alpha}$ (and thus in $C^{\infty}$ by parabolic Schauder estimates and bootstrapping), the second term is uniformly bounded. Thus, exponential (polynomial) decay of the $L^{2}$ norm give exponential (polynomial) decay of the $C^{2,\alpha}$ norm as well. 
\end{proof}

\begin{remark}
The assumption in Theorem \ref{theo:gen-rate-conv} that $u(t)$ converges in $C^{2,\alpha}$ to the constant function $1$ can be weakened to assuming merely that the Yamabe flow converges in $L^N(M,g_0)$. Indeed, it is possible to show that the latter already implies the flow has a smooth limit to which it converges in $C^{2,\alpha}$. The $L^N$ convergence is equivalent to saying that the flow converges in the Ebin $L^2$ metric on the space of Riemannian metrics, restricted to the conformal class \cite[\S4]{ClarkeR2011}.
\end{remark}


\section{Slowly converging Yamabe flows}\label{sec:slow-flow}

In this section, we show that given a non-integrable critical point $g_{\infty}$ which satisfies a particular hypothesis, then there exists a Yamabe flow $g(t)$ so that $g(t)$ converges to $g_{\infty}$ at exactly a polynomial rate. This shows that the conditions in Theorem \ref{theo:gen-rate-conv} are nearly sharp. We will do so by modifying the arguments
of Adams--Simon
\cite{AdamsSimon} to the parabolic setting 
(in \cite[\S 6]{AdamsSimon}, the authors remark that their results should extend to the parabolic setting, but this requires some serious work.
Moreover, the
Yamabe functional does not completely fit into 
their
framework 
because of the volume normalization term).

\subsection{Projecting the Yamabe flow with estimates}

The goal of this subsection is to obtain an equivalent
formulation on the Yamabe flow in terms of {\it two}
flows: one taking place on the finite-dimensional space
$\Lambda_0$ (the ``kernel-projected flow"), and the other on the infinite-dimensional
complement $\Lambda_0^\perp$ (the ``kernel orthogonal-projected flow"). The a priori estimates
of Proposition \ref{prop:slow-decay-rewrite-two-comp} make
this possible.

The next lemma will provide a function which approximately solves the Yamabe flow. The remaining parts of this section will be devoted to perturbing it to an exact solution of the flow. 
Here and in the sequel we will always use $f'(t)$ to denote the time derivative of a function $f(t)$. The constant $T$ should be thought of as a large, but fixed, parameter. Because none of the constants in the bounds that we will derive in Proposition \ref{prop:slow-decay-rewrite-two-comp} depend implicitly on $T$, we will be allowed to take $T$ large in the final step of the proof of Theorem \ref{theo:gen-slow-conv}.

\begin{lemm}\label{lemm:poly-decay-def-varphi}
Assume that $g_{\infty}$ satisfies $AS_{p}$ as defined in Definition \ref{defi:ASp}, i.e., $F_{p}|_{\SS^{k-1}}$ achieves a positive maximum for some point $\hat v$ in the unit sphere $\SS^{k-1}\subset \Lambda_{0}$. Then, for any $T\geq 0$ fixed, the function
\begin{equation}
\label{phitEq}
\varphi (t) :=\varphi(t,T) =  (T+t)^{-\frac{1}{p-2}} \left( \frac{2(N-2)}{p(p-2)F_{p}(\hat v) } \right)^{\frac{1}{p-2}} \hat v
\end{equation}
solves $2(N-2) \varphi' + D F_{p}(\varphi) =0$. 
\end{lemm}
\begin{proof}
Assume that $F_{p}|_{\SS^{k-1}}$ achieves a positive maximum at $\hat v$.
Then for any $\lambda \in \RR$,
\begin{equation}
D F_{p}(\lambda \hat v) = p |\lambda|^{p-1} F_{p}(\hat v) \hat v.
\end{equation}
The reason for this is that $F_{p}$ is $p$-homogeneous, so it is some function on $\SS^{k-1}$ times $r^{p}$. The $\SS^{k-1}$ part has zero gradient at $\hat v$ by assumption, so the gradient must be radial. The exact form follows from differentiating the $r^{p}$ part along with scaling. Thus, 
\begin{align*}
D F_{p}(\varphi) & = p (T+t)^{-1-\frac{1}{p-2}}  \left( \frac{2(N-2)}{p(p-2)F_{p}(\hat v) } \right)^{1+\frac{1}{p-2}} F_{p}(\hat v) \hat v\\
& = \frac{2(N-2)}{p-2} (T+t)^{-1} \varphi(t).
\end{align*}
Since
$
\varphi'(t)  = -\frac{1}{p-2} (T+t)^{-1}\varphi(t),
$
the claim follows. 
\end{proof}


In the next result and subsequently in this section we will always denote by $\Vert f(t) \Vert_{C^{k,\alpha}}$ the parabolic $C^{k,\alpha}$ norm on $(t,t+1)\times M$. 
More precisely, for $\alpha \in (0,1)$, we define the seminorm
\begin{equation*}
\vert f(t)\vert_{C^{0,\alpha}} = \sup_{\substack{(s_{i},x_{i})\in(t,t+1)\times M \\  (s_{1},x_{1}) \not = (s_{2},t_{2}) }} \frac{|f(s_{1},x_{1}) - f(s_{2},x_{2})|}{(d_{g_{\infty}}(x_{1},x_{2})^{2} + |t_{1}-t_{2}|)^{\frac \alpha 2}},
\end{equation*}
and for $k\geq 0$ and $\alpha \in (0,1)$, we define the norm
\begin{equation}\label{eq:hold-norm}
\Vert f(t) \Vert_{C^{k,\alpha}} = \sum_{|\beta| + 2j \leq k} \sup_{(t,t+1)\times M} \vert D^{\beta}_{x} D^{j}_{t}f\vert +  \sum_{|\beta| + 2j = k} \vert  D^{\beta}_{x} D^{j}_{t}f \vert_{C^{0,\alpha}},
\end{equation}
where the norm and derivatives in the sum are taken with respect to $g_{\infty}$. When we mean an alternative norm, we will \emph{always} indicate the domain. 

The reason that we have chosen these norms is that they will be needed to close the fixed-point iteration argument in Proposition \ref{prop:contract-map}; showing that a certain map is a contraction map will use parabolic Schauder estimates (shown in Lemma \ref{lemm:ker-orthog-prob}), which require the chosen norms. As such, we will use these norms throughout this section. 

First, we prove a preliminary lemma, which allows us to estimate the difference between $D\cY(u)$ 
and the term $u^{2-N}D\cY(u)$ that appears on the right hand side of the evolution 
equation \eqref{YFMainEq} for the conformal factor under Yamabe flow. This will allow us
to reduce this evolution equation to a gradient flow.
Intuitively, this is clear since $u$ is approximately $1$, but
the point is to show that the difference 
$(u^{2-N}-1)D\cY(u)$ is sufficiently small with respect to certain
weighted norms that will be used in the proof of the
contraction mapping argument (\S\ref{ConstructionSubSec}).

\begin{lemm}\label{lemm:DYu-est}
There exists $T_{0} >0$, $\epsilon_{0} >0$ and $c > 0$ all depending on $g_{\infty}$ and $\hat v$ so that the following holds: 
Fix $T > T_{0}$.
Then, for $\varphi(t)$ as in Lemma \ref{lemm:poly-decay-def-varphi} and $w\in C^{2,\alpha}(M\times[0,\infty))$, defining $u:= \Psi(\varphi+w^{\T}) + w^{\perp}$,
then the term
\begin{equation*}
E_{0}^{\T}(w) : = \proj_{\Lambda_{0}}( D\cY(u) u^{2-N} - D\cY(u))
\end{equation*}
satisfies
\begin{equation*}
\left\{
\begin{split}
& \Vert E_{0}^{\T}(w) \Vert_{C^{0,\alpha}} \leq c ( (T+t)^{-1-\frac {1}{p-2}}+\Vert w^{\T} \Vert_{C^{0,\alpha}}^{p-1} + \Vert w^{\perp}\Vert_{C^{2,\alpha}} ) ( (T+t)^{-\frac {1}{p-2}}+ \Vert w \Vert_{C^{2,\alpha}} )\\
& \Vert E_{0}^{\T}(w_{1}) - E_{0}^{\T}(w_{2}) \Vert_{C^{0,\alpha}} \leq c ((T+t)^{-1-\frac {1}{p-2}}+\Vert w_{1}^{\T} \Vert_{C^{0,\alpha}}^{p-1}  + \Vert w_{2}^{\T} \Vert_{C^{0,\alpha}}^{p-1} +  \Vert w_{1}^{\perp}\Vert_{C^{2,\alpha}} + \Vert w_{2}^{\perp} \Vert_{C^{2,\alpha}} )\\
 & \qquad\qquad\qquad\qquad\qquad\qquad\qquad\qquad\qquad\qquad\qquad\qquad\qquad\qquad\qquad \times \Vert w_{1}-w_{2} \Vert_{C^{2,\alpha}}\\
&{ \qquad \qquad\qquad \qquad \qquad \qquad + c ((T+t)^{-\frac {1}{p-2}}+ \Vert w_{1} \Vert_{C^{2,\alpha}} + \Vert w_{2} \Vert_{C^{2,\alpha}} )(\Vert w_{1}^{\T} \Vert_{C^{0,\alpha}}^{p-2} + \Vert w_{2}^{\T} \Vert_{C^{0,\alpha}}^{p-2} )}\\
& \qquad\qquad\qquad\qquad\qquad\qquad\qquad\qquad\qquad\qquad\qquad\qquad\qquad\qquad\qquad \times \Vert w_{1}^{\T} - w_{2}^{\T}\Vert_{C^{0,\alpha}}\\
& \qquad \qquad\qquad \qquad \qquad \qquad + c ((T+t)^{-\frac {1}{p-2}}+ \Vert w_{1} \Vert_{C^{2,\alpha}} + \Vert w_{2} \Vert_{C^{2,\alpha}} ) \Vert w_{1}^{\perp} - w_{2}^{\perp}\Vert_{C^{2,\alpha}}
\end{split}
\right.
\end{equation*}
Identical estimates hold for $E_{0}^{\perp}(w) : = \proj_{\Lambda_{0}^{\perp}}( D\cY(u) u^{2-N} - D\cY(u))$. Here, we are using the parabolic H\"older norms on $(t,t+1)\times M$ as defined above; the bounds hold for each $t\geq0$ fixed, with the constants independent of $T$ and $t$.
\end{lemm}
\begin{proof}
First,
\begin{align*}
 u^{2-N}& = 1 + \int_{0}^{1}\frac{d}{ds}\Bigg[ \Big(  1 + s\underbrace{(\varphi + w^{\T}+ \Phi(\varphi + w^{\T}) + w^{\perp}))}_{:=\phi}\Big)^{2-N} \Bigg]ds\\
& = 1 + (2-N) \int_{0}^{1} \left(  1 + s\phi \right)^{1-N}\phi ds.
\end{align*}
So, letting $E_{0}(w) : = D\cY(u)u^{2-N} - D\cY(u)$, we have that 
\begin{equation}\label{eq:E0-bounds}\begin{split}
 \Vert E_{0}(w) \Vert_{C^{0,\alpha}} & \leq c \Vert D\cY(u)\Vert_{C^{0,\alpha}}( \Vert \varphi \Vert_{C^{0,\alpha}} + \Vert w^{\T}\Vert_{C^{0,\alpha }}+ \Vert \Phi(\varphi + w^{\T})\Vert_{C^{0,\alpha}} + \Vert w^{\perp}\Vert_{C^{0,\alpha}} )\\
 & \leq c \Vert D\cY(u) \Vert_{C^{0,\alpha}}( (T+t)^{-\frac {1}{p-2}}+ \Vert w^{\T}\Vert_{C^{0,\alpha }} + \Vert w^{\perp}\Vert_{C^{0,\alpha}} )
\end{split}\end{equation}
We have used the fact that $\Phi(0)=0$ and $\Phi: \Lambda_{0}\to C^{2,\alpha}(M,g_{\infty})$ is a differentiable map. 
Taylor's theorem shows that for $\psi_{s,r} : = 1 +r( \varphi + w^{\T} + \Phi(\varphi + w^{\T}) + sw^{\perp})$,
\begin{align*}
D\cY(u) & = D\cY(\Psi(\varphi+w^{\T})) + \int_{0}^{1} D^{2}\cY(\psi_{s,1}) [w^{\perp}] ds\\
& = D\cY(\Psi(\varphi+w^{\T})) - 2(N-2) \cL_{\infty} w^{\perp} \\
& \qquad + \int_{0}^{1}\int_{0}^{s} D^{3}\cY(\psi_{s,\tilde s})[w^{\perp}, \varphi + w^{\T} + \Phi(\varphi + w^{\T}) + sw^{\perp} ] d\tilde s ds\\
& = \proj_{\Lambda_{0}} D\cY(\Psi(\varphi+w^{\T})) - 2(N-2) \cL_{\infty} w^{\perp} \\
& \qquad + \int_{0}^{1}\int_{0}^{s} D^{3}\cY(\psi_{s,\tilde s})[w^{\perp}, \varphi + w^{\T} + \Phi(\varphi + w^{\T}) + sw^{\perp} ] d\tilde s ds\\
& = DF(\varphi+w^{\T}) - 2(N-2) \cL_{\infty} w^{\perp} \\
& \qquad + \int_{0}^{1}\int_{0}^{s} D^{3}\cY(\psi_{s,\tilde s})[w^{\perp}, \varphi + w^{\T} + \Phi(\varphi + w^{\T}) + sw^{\perp} ] d\tilde s ds.
\end{align*}
In the last line, we used the bound \eqref{eq:app-D3bds} on $D^{3}\cY$ discussed in Appendix \ref{sec:app}. Now, observe that $DF(0) = D^{2}F(0) = \dots = D^{p-1}F(0)=0$, by definition of $p$, the order of integrability. As such, Taylor's theorem shows that
\begin{equation*}
\Vert DF(\varphi+w^{\T}) \Vert_{C^{0,\alpha}} \leq c \Vert \varphi + w^{\perp}\Vert_{C^{0,\alpha}}^{p-1} \leq c((T+t)^{-1-\frac{1}{p-2}} +\Vert w^{\T} \Vert_{C^{0,\alpha}}^{p-1}). 
\end{equation*}
Bounding the other two terms in the above expression for $D\cY(u)$ 
(using the bound \eqref{eq:app-D3bds} on $D^3\cY$ discussed in Appendix A), we have that
%
\begin{equation}\label{DYAPrioriEstEq}
\Vert D\cY(u)\Vert_{C^{0,\alpha}}\leq c((T+t)^{-1-\frac{1}{p-2}} +\Vert w^{\T} \Vert_{C^{0,\alpha}}^{p-1} + \Vert w^{\perp}\Vert_{C^{2,\alpha}}). 
\end{equation}
We define
\begin{equation*}
E_{0}^{\T}(w) : = \proj_{\Lambda_{0}} E_{0}(w), \qquad E_{0}^{\perp}(w) : = \proj_{\Lambda_{0}^{\perp}} E_{0}(w).
\end{equation*}
The asserted bounds for $E_0^{\T}(w)$  follow from the bound \eqref{eq:E0-bounds}
on $E_0(w)$, 
the estimate \eqref{DYAPrioriEstEq}, and the continuity of
the map $\proj_{\Lambda_{0}}:C^{0,\alpha}(M,g_{\infty}) \to \Lambda_{0}$:
\begin{align*}
\Vert \proj_{\Lambda_{0}} f\Vert_{C^{0,\alpha}(M,g_{\infty})} & \leq c \Vert \proj_{\Lambda_{0}} f\Vert_{L^{2}(M,g_{\infty})} \\
& \leq c  \Vert  f\Vert_{L^{2}(M,g_{\infty})}\\
& \leq c\Vert f \Vert_{C^{0,\alpha}(M,g_{\infty})},
\end{align*} 
where the first inequality follows because of the finite dimensionality of $\Lambda_{0}$. Note that this is a spacial bound bound, so it does not include the $t$-H\"older norm, but the desired space-time norm bound follows easily from the bound: if $f$ is time dependent,
\begin{align*}
\Vert (\proj_{\Lambda_{0}} f)(t_{1}) - (\proj_{\Lambda_{0}} f)(t_{2})\Vert_{C^{0,\alpha}(M,g_{\infty})}  & = \Vert \proj_{\Lambda_{0}}( f(t_{1}) - f(t_{2})) \Vert_{C^{0,\alpha}(M,g_{\infty})}\\
& \leq c \Vert  f(t_{1}) - f(t_{2}) \Vert_{C^{0,\alpha}(M,g_{\infty})}
\end{align*}
Dividing by $|t_{1}-t_{2}|^{\frac{\alpha}{2}}$ and taking the supremum over all such $t_{1},t_{2} \in (t,t+1)$, the asserted bound follows. The bound for $E_{0}^{\T}(w_{1})-E^{\T}_{0}(w_{2})$ 
follows similarly. 
This, combined with the estimate
on $E_0(w)$  \eqref{eq:E0-bounds} together with the estimate \eqref{DYAPrioriEstEq},
implies by the triangle inequality 
that the identical estimates hold also for $E_{0}^{\perp}(w)$. 
\end{proof}

The next result reduces the Yamabe flow to two flows, one
on $\Lambda_0$ and the other on $\Lambda_0^\perp$.

\begin{prop}\label{prop:slow-decay-rewrite-two-comp}
There exists $T_{0} >0$, $\epsilon_{0} >0$ and $c > 0$ all depending on $g_{\infty}$ and $\hat v$ so that the following holds: 
Fix $T > T_{0}$.
Then, for $\varphi(t)$ as in Lemma \ref{lemm:poly-decay-def-varphi} and $w\in C^{2,\alpha}(M\times[0,\infty))$, there are functions $E^{\T}(w)$ and $E^{\perp}(w)$ so that $u:= \Psi(\varphi+w^{\T}) + w^{\perp}$ is a solution to the Yamabe flow if and only if 
\begin{align}
\label{kpfEq} 2(N-2)(w^{\T})' + D^{2} F_{p}(\varphi) w^{\T} & = E^{\T}(w)\\
\label{kopfEq} (w^{\perp}) ' - \cL_{\infty}w^{\perp} & = E^{\perp}(w).
\end{align}
Here, as long as $\Vert w\Vert_{C^{2,\alpha}} \leq \epsilon_{0}$, the error terms $E^{\T}$ and $E^{\perp}$ satisfy 
\begin{equation*}
\left\{
\begin{split}
& \Vert E^{\T}(w)\Vert _{C^{0,\alpha}} \leq    c ( (T+t)^{-1-\frac {1}{p-2}}+\Vert w^{\T} \Vert_{C^{0,\alpha}}^{p-1} + \Vert w^{\perp}\Vert_{C^{2,\alpha}} ) ( (T+t)^{-\frac {1}{p-2}}+ \Vert w \Vert_{C^{2,\alpha}} ) \\
&  \qquad \qquad\qquad \qquad \qquad \qquad  + c(T+t)^{-\frac{p}{p-2}} + c  (T+t)^{-\frac{p-1}{p-2}} \Vert w^{\T}\Vert _{C^{0,\alpha}}  + c(T+t)^{-\frac{p-3}{p-2}}\Vert w^{\T}\Vert ^{2}_{C^{0,\alpha}} \\
&  \qquad \qquad\qquad \qquad \qquad \qquad  + c\Vert w^{\T}\Vert _{C^{0,\alpha}}^{p-1}   + c ((T+t)^{-\frac{1}{p-2}} + \Vert w \Vert_{C^{2,\alpha}})\Vert w^{\perp}\Vert _{C^{2,\alpha}} \\
& \Vert E^{\T}(w_{1}) - E^{\T}(w_{2})\Vert _{C^{0,\alpha}} \leq c ((T+t)^{-1-\frac {1}{p-2}}+\Vert w_{1}^{\T} \Vert_{C^{0,\alpha}}^{p-1}  + \Vert w_{2}^{\T} \Vert_{C^{0,\alpha}}^{p-1} +  \Vert w_{1}^{\perp}\Vert_{C^{2,\alpha}} + \Vert w_{2}^{\perp} \Vert_{C^{2,\alpha}} )\\
 & \qquad\qquad\qquad\qquad\qquad\qquad\qquad\qquad\qquad\qquad\qquad\qquad\qquad\qquad\qquad \times \Vert w_{1}-w_{2} \Vert_{C^{2,\alpha}}\\
&{ \qquad \qquad\qquad \qquad \qquad \qquad + c ((T+t)^{-\frac {1}{p-2}}+ \Vert w_{1} \Vert_{C^{2,\alpha}} + \Vert w_{2} \Vert_{C^{2,\alpha}} )(\Vert w_{1}^{\T} \Vert_{C^{0,\alpha}}^{p-2} + \Vert w_{2}^{\T} \Vert_{C^{0,\alpha}}^{p-2} )}\\
& \qquad\qquad\qquad\qquad\qquad\qquad\qquad\qquad\qquad\qquad\qquad\qquad\qquad\qquad\qquad \times \Vert w_{1}^{\T} - w_{1}^{\T}\Vert_{C^{0,\alpha}}
& \qquad \qquad\qquad \qquad \qquad \qquad + c ((T+t)^{-\frac {1}{p-2}}+ \Vert w_{1} \Vert_{C^{2,\alpha}} + \Vert w_{2} \Vert_{C^{2,\alpha}} )(\Vert w_{1}^{\T} \Vert_{C^{0,\alpha}}^{p-2} + \Vert w_{2}^{\T} \Vert_{C^{0,\alpha}}^{p-2} )\\
& \qquad \qquad\qquad \qquad \qquad \qquad + c ((T+t)^{-\frac {1}{p-2}}+ \Vert w_{1} \Vert_{C^{2,\alpha}} + \Vert w_{2} \Vert_{C^{2,\alpha}} ) \Vert w_{1}^{\perp} - w_{1}^{\perp}\Vert_{C^{2,\alpha}}
& \qquad \qquad\qquad\qquad \qquad \qquad + c ( \Vert w_{1}^{}\Vert _{C^{2,\alpha}} + \Vert w_{2}^{}\Vert _{C^{2,\alpha}}) \Vert w_{1}^{\perp} - w^{\perp}_{2}\Vert _{C^{2,\alpha}}\\
& \qquad \qquad\qquad\qquad \qquad \qquad + c\big((T+t)^{-\frac{p-3}{p-2}} (\Vert w_{1}^{\T}\Vert _{C^{0,\alpha}} +\Vert w_{2}^{\T}\Vert _{C^{0,\alpha}}) + \Vert w_{1}^{\T}\Vert _{C^{0,\alpha}}^{p-2} + \Vert w_{2}^{\T}\Vert _{C^{0,\alpha}}^{p-2} \big)\\
& \qquad\qquad\qquad\qquad\qquad\qquad\qquad\qquad\qquad\qquad\qquad\qquad\qquad\qquad\qquad \times \Vert w^{\T}_{1}-w^{\T}_{2}\Vert _{C^{0,\alpha}}\\
& \qquad \qquad\qquad\qquad \qquad \qquad +  c (T+t)^{-\frac{p-1}{p-2}} \Vert w^{\T}_{1} - w^{\T}_{2}\Vert _{C^{0,\alpha}}.
\end{split}
\right.
\end{equation*}
and
\begin{equation*}
\left\{
\begin{split}
& \Vert E^{\perp}(w)\Vert _{C^{0,\alpha}}  \leq  c ( (T+t)^{-1-\frac {1}{p-2}}+\Vert w^{\T} \Vert_{C^{0,\alpha}}^{p-1} + \Vert w^{\perp}\Vert_{C^{2,\alpha}} ) ( (T+t)^{-\frac {1}{p-2}}+ \Vert w \Vert_{C^{2,\alpha}} )\\
& \qquad \qquad \qquad \qquad \qquad\qquad + c ((T+t)^{-\frac{1}{p-2}} + \Vert w\Vert _{C^{2,\alpha}}) \Vert w^{\perp}\Vert _{C^{2,\alpha}} \\
& \qquad \qquad \qquad \qquad \qquad\qquad + c((T+t)^{-\frac{1}{p-2}} + \Vert w\Vert _{C^{2,\alpha}})((T+t)^{-1-\frac{1}{p-2}} + \Vert w'\Vert_{C^{0,\alpha}})\\
& \Vert E^{\perp}(w_{1}) - E^{\perp}(w_{2})\Vert _{C^{0,\alpha}}  \leq c ((T+t)^{-1-\frac {1}{p-2}}+\Vert w_{1}^{\T} \Vert_{C^{0,\alpha}}^{p-1}  + \Vert w_{2}^{\T} \Vert_{C^{0,\alpha}}^{p-1} +  \Vert w_{1}^{\perp}\Vert_{C^{2,\alpha}} + \Vert w_{2}^{\perp} \Vert_{C^{2,\alpha}} )\\
 & \qquad\qquad\qquad\qquad\qquad\qquad\qquad\qquad\qquad\qquad\qquad\qquad\qquad\qquad\qquad \times \Vert w_{1}-w_{2} \Vert_{C^{2,\alpha}}\\
&{ \qquad \qquad\qquad \qquad \qquad \qquad + c ((T+t)^{-\frac {1}{p-2}}+ \Vert w_{1} \Vert_{C^{2,\alpha}} + \Vert w_{2} \Vert_{C^{2,\alpha}} )(\Vert w_{1}^{\T} \Vert_{C^{0,\alpha}}^{p-2} + \Vert w_{2}^{\T} \Vert_{C^{0,\alpha}}^{p-2} )}\\
& \qquad\qquad\qquad\qquad\qquad\qquad\qquad\qquad\qquad\qquad\qquad\qquad\qquad\qquad\qquad \times \Vert w_{1}^{\T} - w_{1}^{\T}\Vert_{C^{0,\alpha}}
& \qquad \qquad\qquad \qquad \qquad \qquad + c ((T+t)^{-\frac {1}{p-2}}+ \Vert w_{1} \Vert_{C^{2,\alpha}} + \Vert w_{2} \Vert_{C^{2,\alpha}} )(\Vert w_{1}^{\T} \Vert_{C^{0,\alpha}}^{p-2} + \Vert w_{2}^{\T} \Vert_{C^{0,\alpha}}^{p-2} )\\
& \qquad \qquad\qquad \qquad \qquad \qquad + c ((T+t)^{-\frac {1}{p-2}}+ \Vert w_{1} \Vert_{C^{2,\alpha}} + \Vert w_{2} \Vert_{C^{2,\alpha}} ) \Vert w_{1}^{\perp} - w_{1}^{\perp}\Vert_{C^{2,\alpha}}
& \qquad \qquad \qquad \qquad \qquad\qquad + c(\Vert w_{1}\Vert _{C^{2,\alpha}} + \Vert w_{2}\Vert _{C^{2,\alpha}}) \Vert w_{1}^{\perp} - w_{2}^{\perp}\Vert _{C^{2,\alpha}} \\
& \qquad \qquad \qquad \qquad \qquad\qquad + c((T+t)^{-\frac{1}{p-2}} + \Vert w_{1}\Vert _{C^{2,\alpha}} + \Vert w_{2}\Vert _{C^{2,\alpha}}) \Vert w_{1}' - w_{2}'\Vert _{C^{0,\alpha}} \\
& \qquad \qquad \qquad \qquad \qquad\qquad + c((T+t)^{-1-\frac{1}{p-2}} + \Vert w_{1}'\Vert _{C^{0,\alpha}} + \Vert w_{2}'\Vert _{C^{0,\alpha}})\Vert w_{1}-w_{2}\Vert _{C^{2,\alpha}}.
\end{split}
\right.
\end{equation*}
Here, we are using the parabolic H\"older norms on $(t,t+1)\times M$ as defined above; the bounds hold for each $t\geq0$ fixed, with the constants independent of $T$ and $t$.
\end{prop}
\begin{proof}
Recall that $u$ is a solution to the Yamabe flow if and only if 
\begin{equation}\label{YFMainEq}
(N-2) \frac{\partial u}{\partial t} = - \frac{1}{2}D{\cY}(u) u^{2-N}=  (N+2) u^{2-N} \Delta_{\infty}u - R_{\infty} u^{3-N} + r_{u^{N-2}g_{\infty}}u,
\end{equation}
where, as always in this article, $\cY$ is defined on the {unit volume conformal class} and $D(\cdot)$ is the corresponding constrained differential. We now project the Yamabe flow equation onto $\Lambda_{0}$ and $\Lambda_{0}^{\perp}$, so $u$ solves Yamabe flow if and only if the following two equations are satisfied 
\begin{equation*}\begin{split}
2(N-2)(\varphi+ w^{\T})' & = - \proj_{\Lambda_{0}} \left[ D{\cY}(  1 + \varphi + w^{\T}+ \Phi(\varphi + w^{\T}) + w^{\perp} ) u(t)^{2-N}\right] + E_{0}^{\T}(w)\\
2(N-2)(\Phi(\varphi+w^{\T})+w^{\perp})' & = - \proj_{\Lambda_{0}^{\perp}} \left[ D{\cY}(  1 + \varphi + w^{\T}+ \Phi(\varphi + w^{\T}) + w^{\perp} ) u(t)^{2-N}\right] + E_{0}^{\perp}(w).
\end{split}
\end{equation*}

Here, we emphasize that (as in the previous section) we are considering $D\cY(w)$ as a \emph{function} on $M$, via the $L^{2}(M,g_{\infty})$ pairing. In other words, we are using \eqref{eq:Dy-L2}. 
Now, we claim that we may use Taylor's theorem to show that
\begin{equation*}
\proj_{\Lambda_{0}} D \cY(  1 + \varphi + w^{\T}+ \Phi(\varphi + w^{\T}) + w^{\perp} ) =\proj_{\Lambda_{0}} D {\cY}(  1 + \varphi + w^{\T}+ \Phi(\varphi + w^{\T})) + E_{1}^{\T}(w),
\end{equation*}
with the following bounds
\begin{equation*}
\left\{
\begin{split}
& \Vert E_{1}^{\T}(w)\Vert _{C^{0,\alpha}}  \leq c ((T+t)^{-\frac{1}{p-2}} + \Vert w \Vert_{C^{2,\alpha}})\Vert w^{\perp}\Vert _{C^{2,\alpha}}\\
& \Vert E_{1}^{\T}(w_{1}) - E_{1}^{\T}(w_{2})\Vert _{C^{0,\alpha}}  \leq c ( \Vert w_{1}\Vert _{C^{2,\alpha}} + \Vert w_{2}\Vert _{C^{2,\alpha}}) \Vert w_{1}^{\perp} - w^{\perp}_{2}\Vert _{C^{2,\alpha}}
\end{split}
\right.
\end{equation*}
These follow from the integral form of the remainder in Taylor's theorem. Defing $\psi_{s,r} : = 1+r(
\varphi + w^{\T}+\Phi(\varphi + w^{\T}) + s w^{\perp})$, we have
\begin{align*}
E_{1}^{\T}(w) & = \int_{0}^{1}\frac{d}{ds} \proj_{\Lambda_{0}} D\cY (1+\varphi + w^{\T}+\Phi(\varphi + w^{\T}) + s w^{\perp})ds\\
& = \int_{0}^{1} \proj_{\Lambda_{0}}\frac{d}{ds}  D\cY (\psi_{s,1})ds\\
& = \int_{0}^{1}\proj_{\Lambda_{0}} D^{2}\cY(\psi_{s,1}^{N-2}g_{\infty})[w^{\perp}] ds\\
& = \int_{0}^{1}\proj_{\Lambda_{0}} D^{2}\cY(g_{\infty})[w^{\perp}] ds + \int_{0}^{1} \int_{0}^{s} \proj_{\Lambda_{0}} \frac{d}{d\tilde s} D^{2}\cY(\psi_{\tilde s}^{N-2}g_{\infty})[w^{\perp}]d\tilde sds\\
& = -2(N-2)\int_{0}^{1}\proj_{\Lambda_{0}}\cL_{\infty} w^{\perp} ds + \int_{0}^{1} \int_{0}^{s} \proj_{\Lambda_{0}} \frac{d}{d\tilde s} D^{2}\cY(\psi_{s,\tilde s}^{N-2}g_{\infty})[w^{\perp}]d\tilde sds\\
& =  \int_{0}^{1} \int_{0}^{s} \proj_{\Lambda_{0}} \frac{d}{d\tilde s} D^{2}\cY|(\psi_{s,\tilde s}^{N-2}g_{\infty})[w^{\perp}]d\tilde sds\\
& =  \int_{0}^{1} \int_{0}^{s} \proj_{\Lambda_{0}} D^{3}\cY(\psi_{s,\tilde s}^{N-2}g_{\infty})[w^{\perp},\varphi + w^{\T} +\Phi(\varphi + w^{\T}) + sw^{\perp}]d\tilde sds.
\end{align*}
The $C^{0,\alpha}$ norm of $ D^{3}\cY(\psi_{s,\tilde s}^{N-2}g_{\infty})[w^{\perp},\varphi + w^{\T} +\Phi(\varphi + w^{\T}) + sw^{\perp}]$ is uniformly bounded by 
\begin{equation*}
c ((T+t)^{-\frac{1}{p-2}} + \Vert w \Vert_{C^{2,\alpha}})\Vert w^{\perp}\Vert _{C^{2,\alpha}}
\end{equation*}
(as long as we choose $T \geq T_{0}$ large enough, and $\Vert w \Vert_{C^{2,\epsilon}}\leq \epsilon_{0}$ small enough to ensure that $\psi_{s,\tilde{s}}$ is sufficiently close to $1$ in $C^{2,\alpha}$). This is discussed in the end of 
Appendix A, along with the corresponding bound for $E_{1}^{\T}(w_{1}) - E_{1}^{\T}(w_{2})$.

Recall that $F(v) := \cY(\Psi(v))$, and
using the Lyapunov--Schmidt reduction (Proposition \ref{prop:LSred})
\begin{equation*}
\proj_{\Lambda_{0}} D{\cY}(  1 + \varphi + w^{\T}+ \Phi(\varphi + w^{\T})) = D F(\varphi+w^{\T}).
\end{equation*}
 Furthermore, by analyticity (Lemma \ref{AnalyticLemma} and Proposition \ref{prop:LSred}) $D F$ has a convergent power series representation around $0$ with lowest order term of order $p-1$. Thus,
 as long as $\varphi + w^{\T}$ is small enough, we may write
\begin{equation*}
D F(\varphi + w^{\T}) = D F(\varphi)  + D^{2} F(\varphi) (w^{\T}) + E^{\T}_{2}(w^{\T}).
\end{equation*}
where
\begin{equation*}
\left\{
\begin{split}
& \Vert E^{\T}_{2}(w^{\T})\Vert _{C^{0,\alpha}}  \leq  c((T+t)^{-\frac{p-3}{p-2}} + \Vert w^{\T}\Vert _{C^{0,\alpha}}^{p-3}) \Vert w^{\T}\Vert _{C^{0,\alpha}}^{2}\\
& \Vert E^{\T}_{2}(w_{1}^{\T}) - E^{\T}_{2}(w_{2}^{\T})\Vert _{C^{0,\alpha}}  \leq c\big((T+t)^{-\frac{p-3}{p-2}} (\Vert w_{1}^{\T}\Vert _{C^{0,\alpha}} +\Vert w_{2}^{\T}\Vert _{C^{0,\alpha}}) + \Vert w_{1}^{\T}\Vert _{C^{0,\alpha}}^{p-2} + \Vert w_{2}^{\T}\Vert _{C^{0,\alpha}}^{p-2} \big) \\
& \qquad\qquad\qquad\qquad\qquad\qquad\qquad\qquad\qquad\qquad\qquad\qquad\qquad\qquad\qquad \times \Vert w_{1}^{\T}-w_{2}^{\T}\Vert _{C^{0,\alpha}}.
\end{split}
\right.
\end{equation*}
We prove this in the case that $\Lambda_{0}$ has dimension equal to one, namely for $k=1$; the higher dimensional case follows from a similar argument using multi-index notation. We have that
\begin{equation*}
DF (z) = \sum_{j=p-1}^{\infty}c_{j}z^{j}.
\end{equation*}
Thus
\begin{align*}
\Vert E_{2}^{\T}(w^{\T})\Vert_{C^{0,\alpha}} & = \left \Vert \sum_{j=p-1}^{\infty} c_{j} \left[(\varphi +w^{\T})^{j}  - \varphi^{j} - j \varphi^{j-1} w^{\T}\right] \right\Vert_{C^{0,\alpha}}\\
& = \left \Vert \sum_{j=p-1}^{\infty} \sum_{l=2}^{j} c_{j} \binom j l \varphi^{j-l} (w^{\T})^{l} \right\Vert_{C^{0,\alpha}}\\
& \leq  \sum_{j=p-1}^{\infty} \sum_{l=2}^{j} |c_{j}| \binom j l  \Vert \varphi\Vert_{C^{0,\alpha}}^{j-l} \Vert w^{\T}\Vert_{C^{0,\alpha}}^{l}\\
& \leq \Vert w^{\T}\Vert_{C^{0,\alpha}}^{2}  \sum_{j=p-1}^{\infty} \sum_{l=2}^{j} |c_{j}| \binom j l \left(  \Vert \varphi\Vert_{C^{0,\alpha}}^{j-2} + \Vert w^{\T}\Vert_{C^{0,\alpha}}^{j-2} \right)\\
& \leq \Vert w^{\T}\Vert_{C^{0,\alpha}}^{2}  \sum_{j=p-1}^{\infty} |c_{j}| 2^{j} \left(  \Vert \varphi\Vert_{C^{0,\alpha}}^{j-2} + \Vert w^{\T}\Vert_{C^{0,\alpha}}^{j-2} \right)\\
& \leq 2^{p-1} \Vert \varphi\Vert_{C^{0,\alpha}}^{p-3}  \Vert w^{\T}\Vert_{C^{0,\alpha}}^{p-1}  \sum_{j=p-1}^{\infty} |c_{j}| 2^{j+1-p} \Vert \varphi\Vert_{C^{0,\alpha}}^{j+1-p} \\
& \qquad + 2^{p-1} \Vert w^{\T}\Vert_{C^{0,\alpha}}^{p-1}  \sum_{j=p-1}^{\infty} |c_{j}| 2^{j+1-p} \Vert w^{\T}\Vert_{C^{0,\alpha}}^{j+1-p}.
\end{align*}
Because $D^{p-1}F(z)$ has an absolutely convergent power series for every $z$ sufficiently close to $0$, choosing $\epsilon_{0}$ small enough, $T_{0}$ large enough, and using Lemma \ref{lemm:poly-decay-def-varphi}, we may guarantee that both sums are bounded above by $1$. The asserted bound on $ E_{2}^{\T}(w^{\T})$ follows. A similar argument yields the other bound. 

Thus, the above arguments show that the $\Lambda_{0}$-component of the Yamabe flow may be written as 
\begin{equation*}
2(N-2)(\varphi' + (w^{\T})' )= - D F(\varphi) -  D^{2} F(\varphi) (w^{\T}) +E^{\T}_{1}(w) - E^{\T}_{2}(w^{\T}).
\end{equation*}
Now, expanding $F$ in a power series, $F = F(0) + \sum_{j=p}^{\infty}F_{j}$, we may write the above expression as
\begin{equation*}
2(N-2)(\varphi' + (w^{\T})' )= -D F_{p}(\varphi) - D^{2} F_{p}(\varphi) (w^{\T}) + \underbrace{E^{\T}_{1}(w) - E^{\T}_{2}(w^{\T}) + E^{\T}_{3}(w)}_{:=E^{\T}(w)},
\end{equation*}
where 
\begin{equation*}
E_{3}^{\T}(w) = \sum_{j\geq p+1} (D F_{j}(\varphi) + D^{2} F_{j}(\varphi) w^{\T}).
\end{equation*}
By analyticity, this converges in, e.g.\ $C^{0,\alpha}$, for $\Vert \varphi\Vert _{C^{2,\alpha}}$ and $\Vert w\Vert _{C^{2,\alpha}}$ small enough. Because each term in the sum is a homogeneous polynomial,
and using the formula for $\varphi$,
the error is bounded as follows:
\begin{equation*}
\left\{
\begin{split}
& \Vert E_{3}^{\T}(w)\Vert _{C^{0,\alpha}}  \leq c((T+t)^{-\frac{p}{p-2}} +  (T+t)^{-\frac{p-1}{p-2}} \Vert w^{\T}\Vert _{C^{0,\alpha}})\\
& \Vert E_{3}^{\T}(w_{1}) - E_{3}^{\T}(w_{2})\Vert _{C^{0,\alpha}}  \leq c (T+t)^{-\frac{p-1}{p-2}} \Vert w^{\T}_{1} - w^{\T}_{2}\Vert _{C^{0,\alpha}}.
\end{split}
\right.
\end{equation*}
Therefore, by Lemma \ref{lemm:poly-decay-def-varphi}, $w^{\T}$ needs to satisfy the equation
\begin{equation*}
2(N-2)(w^{\T})' + D^{2} F_{p}(\varphi)w^{\T} = E^{\T}(w),
\end{equation*}
where $E^{\T}(w)$ satisfies the asserted bounds. 

Now we turn to 
the $\Lambda_{0}^{\perp}$ portion of 
the Yamabe flow.
First, 
recall that 
by Proposition \ref{prop:LSred},
\begin{equation*}
\proj_{\Lambda_{0}^{\perp}} D{\cY}(\Psi(\varphi+w^{\T})) = 0.
\end{equation*}
Combined with the fact that $D\proj_{\Lambda_{0}^{\perp}} D {\cY}$ at $1$ equals $-2(N-2)\proj_{\Lambda_{0}^{\perp}} \cL_{\infty}$ (this follows because $D$ and $\proj_{\Lambda_{0}^{\perp}}$ commute), we thus may use Taylor's theorem to write (using the fact that $\cL_{\infty}$ is linear)
\begin{equation*}
\proj_{\Lambda_{0}^{\perp}} D{\cY}(\Psi(\varphi+w^{\T}) + w^{\perp}) = -2(N-2)\cL_{\infty} w^{\perp} - E^{\perp}_{1}(w),
\end{equation*}
where
\begin{equation*}
\left\{
\begin{split}
& \Vert E^{\perp}_{1}(w)\Vert _{C^{0,\alpha}}  \leq c ((T+t)^{-\frac{1}{p-2}} + \Vert w\Vert _{C^{2,\alpha}}) \Vert w^{\perp}\Vert _{C^{2,\alpha}} \\
& \Vert E^{\perp}_{1}(w_{1}) - E^{\perp}_{1}(w_{2})\Vert _{C^{0,\alpha}}  \leq c(\Vert w_{1}\Vert _{C^{2,\alpha}} + \Vert w_{2}\Vert _{C^{2,\alpha}}) \Vert w_{1}^{\perp} - w_{2}^{\perp}\Vert _{C^{2,\alpha}} \\
& \qquad \qquad \qquad \qquad \qquad\qquad + c(\Vert w_{1}^{\perp}\Vert _{C^{2,\alpha}} + \Vert w_{2}^{\perp}\Vert _{C^{2,\alpha}}) \Vert w_{1} - w_{2}\Vert _{C^{2,\alpha}} .
\end{split}
\right.
\end{equation*}
To check this, we write
\begin{align*}
\proj_{\Lambda_{0}^{\perp}} & D{\cY}(\Psi(\varphi+w^{\T}) + w^{\perp}) \\
& = \proj_{\Lambda_{0}^{\perp}} D{\cY}(\Psi(\varphi+w^{\T}) ) + \int_{0}^{1}\frac{d}{ds} \proj_{\Lambda_{0}^{\perp}} D{\cY}(\Psi(\varphi+w^{\T}) + sw^{\perp}) ds\\
& = \int_{0}^{1} \proj_{\Lambda_{0}^{\perp}} D^{2}{\cY}(\Psi(\varphi+w^{\T}) + sw^{\perp}) [w^{\perp}] ds\\
& = -2(N-2)\cL_{\infty} w^{\perp} + \int_{0}^{1}\left[  \proj_{\Lambda_{0}^{\perp}}  D^{2}{\cY}(\Psi(\varphi+w^{\T}) + sw^{\perp}) [w^{\perp}] -  \proj_{\Lambda_{0}^{\perp}}  D^{2}\cY(1)[w^{\perp}]  \right] ds.
\end{align*}
Given this, we may control the asserted $C^{0,\alpha}$ bounds by the $C^{2,\alpha}$ norm of $\varphi$ and $w$, by an argument similar to $E_{1}^{\T}$ (the derivative of $\Psi$ is uniformly bounded near $0$ by Proposition \ref{prop:LSred}).

We also consider $\Phi(\varphi+w^{\T})' : = E^{\perp}_{2}(w)$ as an error term, as it satisfies
\begin{equation*}
\left\{
\begin{split}
& \Vert E^{\perp}_{2}(w)\Vert _{C^{0,\alpha}}  \leq c((T+t)^{-\frac{1}{p-2}} + \Vert w\Vert _{C^{2,\alpha}})((T+t)^{-1-\frac{1}{p-2}} + \Vert w'\Vert _{C^{0,\alpha}})\\
& \Vert E^{\perp}_{2}(w_{1}) - E^{\perp}_{2}(w_{2})\Vert _{C^{0,\alpha}}  \leq c((T+t)^{-\frac{1}{p-2}} + \Vert w_{1}\Vert _{C^{2,\alpha}} + \Vert w_{2}\Vert _{C^{2,\alpha}}) \Vert w_{1}' - w_{2}'\Vert _{C^{0,\alpha}} \\
& \qquad \qquad \qquad \qquad \qquad\qquad + c((T+t)^{-1-\frac{1}{p-2}} + \Vert w_{1}'\Vert _{C^{0,\alpha}} + \Vert w_{2}'\Vert _{C^{0,\alpha}})\Vert w_{1}-w_{2}\Vert _{C^{2,\alpha}}.
\end{split}
\right.
\end{equation*}
Here we have used \eqref{eq:LS-est-Dphi} and have controlled $\Vert w \Vert_{L^{2}}$ by the $C^{2,\alpha}$ norm. Thus, the kernel-orthogonal component of Yamabe flow is 
\begin{equation*}
(w^{\perp})' = \cL_{\infty} w^{\perp} + E^{\perp}(w),
\end{equation*}
where $E^{\perp}(w)$ satisfies the asserted bounds. Combining the $\Lambda_{0}$ equation with the $\Lambda_{0}^{\perp}$ equation finishes the proof. 
\end{proof}

\subsection{Solving the kernel-projected flow with polynomial decay estimates}

In this subsection we solve the
kernel-projected flow \eqref{kpfEq}.
First, from the definition of $\varphi$ in \eqref{phitEq} 
and the fact that $D^{2} F_{p}$ is homogeneous of degree $p-2$,
\begin{equation*}
D^{2} F_{p}(\varphi) 
= (T+t)^{-1} \left( \frac{2(N-2)}{p(p-2)F_{p}(\hat v)} \right) D^{2} F_{p}(\hat v).
\end{equation*}
Diagonalize the Hessian term,
and denote by $\mu_{1},\dots,\mu_{k}$ the eigenvalues of 
\begin{equation*}
\frac{2(N-2)}{p(p-2)F_{p}(\hat v)} D^{2} F_{p}(\hat v).
\end{equation*}
Let $e_{i}$ be an orthonormal basis in which this Hessian is diagonalized.
Thus, 
the kernel-projected 
flow 
is equivalent to the following system of ODEs for\footnote{Using, as above, the natural $L^{2}$ inner product on $\Lambda_{0}$ regarded as a subset of $T_{1}[g_{\infty}]_{1}$.}
$v_{i}:=w^{\T}\cdot e_{i}$,
\begin{equation}\label{ODESysEq}
(N-2)v_{i}' + \frac{\mu_{i}}{T+t} v_{i} =E^{\T}_{i}
:=E^{\T}\cdot e_{i}, \quad i=1,\ldots,k.
\end{equation}
Fix 
for the rest of this subsection a number $\gamma$ satisfying
$\gamma \not\in\{\frac{\mu_{1}}{2(N-2)},\ldots,\frac{\mu_{k}}{2(N-2)}\}$.
Define the following weighted norms:
\begin{equation*}
\Vert u \Vert _{C^{0,\alpha}_{\gamma}} : = 
\sup_{t>0}\, [ (T+t)^{\gamma}\Vert u(t)\Vert_{C^{0,\alpha}}],\qquad 
\Vert u\Vert _{C^{0,\alpha}_{1,\gamma}} : = \Vert u\Vert _{C^{0,\alpha}_{\gamma}} + \Vert u'\Vert _{C^{0,\alpha}_{1+\gamma}}.
\end{equation*}
We recall that these H\"older norms are space-time norms on the interval $(t,t+1)\times M$, as defined in \eqref{eq:hold-norm}.

Given $\gamma$ as above, let $\Pi_{0}=\Pi_{0}(\gamma)$ 
denote 
the linear subspace of $\Lambda_{0}$ generated by the eigenvectors of 
$\frac{2(N-2)}{p(p-2)F_{p}(\hat v)} D^{2} F_{p}(\hat v)$ whose eigenvalue, say $\mu$,  satisfies $\mu>2(N-2)\gamma$. Moreover, let $\proj_{\Pi_{0}}:\Lambda_{0}\to\Pi_{0}$ be the corresponding linear projector.

The next lemma concerns the system \eqref{ODESysEq}.
\begin{lemma}\label{lemm:ker-proj-ODE}
For any $T>0$ such that $\Vert E^{\T}\Vert_{C^{0,\alpha}_{1+\gamma}}<\infty$,
there is a unique $u$ with $u(t)\in\Lambda_0, t\in[0,\infty),$ satisfying $\Vert u\Vert _{C^{0}_{\gamma}} <\infty$, 
$\proj_{\Pi_{0}}\left(u(0)\right)= 0$, and such that $v_i:=u\cdot e_i$
solves the system \eqref{ODESysEq}.
Furthermore, we have the bound
\begin{equation*}
\Vert u\Vert _{C^{0,\alpha}_{1,\gamma}} \leq C  \Vert E^{\T}\Vert _{C^{0,\alpha}_{1+\gamma}}.
\end{equation*}
\end{lemma}
\begin{proof}
Letting
\begin{equation*}
w_{j}:=(T+t)^{\frac{\mu_{j}}{2(N-2)} } v_{j},
\end{equation*}
the system \eqref{ODESysEq} is equivalent to 
\begin{equation*}
w_{j}' = \frac{1}{2(N-2)} (T+t)^{\frac{\mu_{j}}{2(N-2)}} E^{\T}_{j}, \quad j=1,\ldots,k.
\end{equation*}

Suppose that $j$ 
is such that
$\gamma > \frac{\mu_{j}}{2(N-2)}$. Then, we claim that we may solve the $j$-th
ODE as 
\begin{equation*}
w_{j}(t) = \alpha_{j} - (2(N-2))^{-1} \int_{t}^{\infty} (T+\tau)^{\frac{\mu_{j}}{2(N-2)}} E^{\T}_{j}(\tau) d\tau.
\end{equation*}
which would give
\begin{equation*}
u_{j}(t) =(T+t)^{-\frac{\mu_{j}}{2(N-2)}}  \alpha_{j}  - (2(N-2))^{-1} (T+t)^{-\frac{\mu_{j}}{2(N-2)}}  \int_{t}^{\infty} (T+\tau)^{\frac{\mu_{j}}{2(N-2)}} E^{\T}_{j}(\tau)d\tau.
\end{equation*}
This amounts to checking that the integral converges under our assumptions on $E^{\T}$:
\begin{equation*}\begin{split}
& \left| (T+t)^{-\frac{\mu_{j}}{2(N-2)}}  \int_{t}^{\infty} (T+\tau)^{\frac{\mu_{j}}{2(N-2)}} E^{\T}_{i}(\tau)d\tau\right| \\
& \leq (T+t)^{-\frac{\mu_{j}}{2(N-2)}}\Vert E^{\T}_{j}\Vert_{C^{0}_{1+\gamma}}  \int_{t}^{\infty} (T+\tau)^{\frac{\mu_{i}}{2(N-2)}-\gamma-1} d\tau\\
& =  
\left(
\gamma - \frac{\mu_{j}}{2(N-2)} 
\right)^{-1} (T+t)^{-\frac{\mu_{j}}{2(N-2)}}\Vert E^{\T}_{i}\Vert _{C^{0}_{1+\gamma}} \left( T+t\right)^{\frac{\mu_{j}}{2(N-2)} - \gamma}\\
& = C_{i} (T+t)^{-\gamma} \Vert E^{\T}_{j}\Vert_{C^{0}_{1+\gamma}}
\end{split}
\end{equation*}
The previous estimate also shows that, since 
by assumption
$\gamma > \frac{\mu_{j}}{2(N-2)}$, to have $\Vert u \Vert_{C^{0,\alpha}_{\gamma}} < \infty$ it must hold that $\alpha_{j} = 0$.

On the other hand, if $\gamma < \frac{\mu_{j}}{2(N-2)}$, we may solve the ODE as
\begin{equation*}
w_{j}(t) = \alpha_{j} + (2(N-2))^{-1} \int_{0}^{t} (T+\tau)^{\frac{\mu_{j}}{2(N-2)}} E^{\T}_{j}(\tau) d\tau.
\end{equation*}
By requiring $\proj_{\Pi_{0}} u(0) = 0$, we see that $\alpha_{j} = 0$. As such, the bounds for $\Vert u_{j}\Vert_{C^{0}_{\gamma}}$ follow from a similar calculation as before.

Combining these two cases proves existence, uniqueness and the $\Vert u\Vert_{C^{0}_{\gamma}}$ bound. It thus remains to show the inequality 
$\Vert u\Vert _{C^{0,\alpha}_{1,\gamma}} \leq C  \Vert E^{\T}\Vert _{C^{0,\alpha}_{1+\gamma}}.
$
By finite dimensionality, the (spatial) $C^{0,\alpha}(M)$-H\"older norms of each basis element in $\Lambda_{0}$ are uniformly bounded. Thus, it remains to show that the desired inequality holds for the H\"older norm in the time direction, along with the same thing for $u'(t)$ (the general space-time norm will then be bounded by the triangle inequality). Suppose that $j$ 
is such that
$\gamma > \frac{\mu_{j}}{2(N-2)}$. Then, we have seen above that
\begin{equation*}
u_{j}(t) = - (2(N-2))^{-1} (T+t)^{-\frac{\mu_{j}}{2(N-2)}}  \int_{t}^{\infty} (T+\tau)^{\frac{\mu_{j}}{2(N-2)}} E^{\T}_{j}(\tau)d\tau.
\end{equation*}
Notice that
\begin{equation*}
u_{j}'(t) = \mu_{j} (T+t)^{-\frac{\mu_{j}}{2(N-2)}-1}  \int_{t}^{\infty} (T+\tau)^{\frac{\mu_{j}}{2(N-2)}} E^{\T}_{j}(\tau)d\tau  - (2(N-2))^{-1} E_{j}^{\T}(t).
\end{equation*}
Thus,
\begin{align*}
\Vert u'_{j}\Vert_{C^{0,\alpha}} & \leq  C \left\Vert  (T+t)^{-\frac{\mu_{j}}{2(N-2)}-1}  \int_{t}^{\infty} (T+\tau)^{\frac{\mu_{j}}{2(N-2)}} E^{\T}_{j}(\tau)d\tau \right\Vert_{C^{1}} + C \Vert E_{j}^{\T}(t) \Vert_{C^{0,\alpha}}\\
& \leq C \left\Vert  (T+t)^{-\frac{\mu_{j}}{2(N-2)}-2}  \int_{t}^{\infty} (T+\tau)^{\frac{\mu_{j}}{2(N-2)}} E^{\T}_{j}(\tau)d\tau \right\Vert_{C^{0}} + C \Vert E_{j}^{\T}(t) \Vert_{C^{0,\alpha}}\\
& \leq C(T+t)^{-1-\gamma} \Vert E_{j}^{\T} \Vert_{C^{0,\alpha}_{1+\gamma}}. 
\end{align*}
We may use the $C^{0}$-bound on $u'_{j}$ to obtain a H\"older estimate for $u_{j}$. From this, the claimed inequality follows.
\end{proof}

\subsection{Solving the kernel-orthogonal projected flow}

Define the weighted norms
$$
\Vert u \Vert_{L^{2}_{q}} = \sup_{t\in [0,\infty)}[(T+t)^{q} \Vert u(t)\Vert_{L^{2}(M)}],
$$ 
where the $L^{2}$ norm is the spatial norm of $u(t)$ on $M$ and is taken with respect to $g_{\infty}$, and
\begin{equation*}
\Vert u \Vert _{C^{2,\alpha}_{q}} = \sup_{t\geq 0}[(T+t)^{q}\Vert u(t)\Vert_{C^{2,\alpha}}],
\end{equation*}
where the H\"older norms are the space-time norms defined in \eqref{eq:hold-norm}. Also, let
\begin{align*}
\Lambda_{\uparrow} & : = \Span\{ \varphi \in C^{\infty}(M) : \cL_{\infty} \varphi + \delta \varphi = 0 ,\delta < 0\},\\
\Lambda_{\downarrow} & : = \overline { \Span \{ \varphi\in C^{\infty}(M) : \cL_{\infty} \varphi + \delta \varphi = 0,\delta > 0\}}^{L^{2}}.
\end{align*}
From the spectral theory of the Laplacian, 
$L^{2}(M,g_{\infty}) = \Lambda_{\uparrow}\oplus\Lambda_{0}\oplus \Lambda_{\downarrow}$  
and 
$\Lambda_{\uparrow}$ and $\Lambda_{0}$ are finite dimensional. 
Write the non-negative integers as an ordered union $\NN = K_{\uparrow}\cup K_{0}\cup K_{\downarrow}$, where the ordering of the indices comes from an ordering of the eigenfunctions of the Laplacian, $\Delta_{g_{\infty}}$ and the partitioning of $\NN$ corresponds to which of $\Lambda_{\downarrow},\Lambda_{0},$ or $\Lambda_{\uparrow}$, the $k$-th eigenfunction of $\Delta_{g_{\infty}}$ lies in. 

\begin{lemma}\label{lemm:ker-orthog-prob}
For any $T>0$ and $q < \infty$ 
such that $\Vert E^\perp\Vert_{L^{2}_{q}}<\infty$,
there is a unique $u(t)$ with $u(t)\in\Lambda^\perp_0, t\in[0,\infty),$ 
satisfying $\Vert u\Vert_{L^{2}_{q}} <\infty$,
$\proj_{\Lambda_{\downarrow}} (u(0)) = 0$,  and
\begin{equation}\label{LemmaFlowEq}
u' = \cL_{\infty}u + E^\perp.
\end{equation}
Furthermore,
$
\Vert u\Vert_{L^{2}_{q}} \leq C \Vert E^{\perp}\Vert_{L^{2}_{q}},
$
and
$\Vert u\Vert_{C^{2,\alpha}_{q}} \leq C \Vert E^{\perp}\Vert_{C^{0,\alpha}_{q}}.
$
\end{lemma}

\begin{proof}
Recall that 
\begin{equation}
(N-2)\cL_{\infty} = (N+2) \Delta_{\infty}+ (N-2)R_{\infty}.
\end{equation}
Let $\varphi_{i}$ be an eigenfunction (with eigenvalue $\delta_{i}$) of $\frac 12 \cL_{\infty}$ which is orthogonal to the kernel $\Lambda_{0}$. 
The flow equation \eqref{LemmaFlowEq} reduces to the system
\begin{equation}
u_{i}' + \delta_{i}u_{i} = E^{\perp}_{i}
\end{equation}
where $E^{\perp}_{i} = \langle E^{\perp},\varphi_{i} \rangle$,
and $u_i=\langle u,\varphi_i\rangle$. This is equivalent to
\begin{equation}
\left( e^{\delta_{i}t }u_{i} \right)' = e^{\delta_{i}t}E^{\perp}_{i}
\end{equation}
Thus, we may represent the components of the solution as 
\begin{equation*}
u_i^{\perp} (t) = \beta_{i} e^{-\delta_{i}t}+ e^{-\delta_{i}t}\int_{0}^{t}e^{\delta_{i}\tau} E^{\perp}_{i}(\tau) d\tau
\end{equation*}
for $i \in K_{\downarrow}$ or
\begin{equation*}
u_i^{\perp} (t) = \beta_{i} e^{-\delta_{i}t}- e^{-\delta_{i}t}\int_{t}^{\infty}e^{\delta_{i}\tau} E^{\perp}_{i}(\tau) d\tau
\end{equation*}
for $i\in K_{\uparrow}$. In particular, we have that
\begin{equation*}
u (t) = \sum_{j\in K_{\downarrow}} \left( \beta_{j} e^{-\delta_{j}t} + e^{-\delta_{j}t}\int_{0}^{t}e^{\delta_{j}\tau} E^{\perp}_{j}(\tau) d\tau \right) \varphi_{j} + \sum_{j\in K_{\uparrow}} \left( \beta_{j} e^{-\delta_{j}t} - e^{-\delta_{j}t}\int_{t}^{\infty}e^{\delta_{j}\tau} E^{\perp}_{j}(\tau) d\tau \right) \varphi_{j}
\end{equation*}
This sum is in an $L^{2}$ sense (but then elliptic regularity guarantees that the sum converges uniformly on compact time intervals). We note that for $i \in K_{\uparrow}$, if $\Vert u \Vert_{L^{2}_{q}} < \infty$ then necessarily $\beta_{i} = 0$. Furthermore, by requiring that $\proj_{\Lambda_{\downarrow}} u(0) = 0$, 
then we also have $\beta_{i}=0$
for $i \in K_{\downarrow}$.

We can write the following integral bounds for $u$:
\begin{align*}
\left\Vert \sum_{j\in K_{\downarrow}} u_{j}(t)\varphi_{j} \right\Vert^{2}_{L^{2}} 
&  \leq \sum_{j\in K_{\downarrow}}\left(\int_{0}^{t}e^{\delta_{j}\left(\tau-t\right)}E_{j}^{\perp}\left(\tau\right)\,d\tau\right)^{2}\\
 & \leq \sum_{j\in K_{\downarrow}}\left(\int_{0}^{t}e^{\delta_{\min}\left(\tau-t\right)}E_{j}^{\perp}\left(\tau\right)\,d\tau\right)^{2} \\
& \leq\left\|\int_{0}^{t}e^{\delta_{\min}\left(\tau-t\right)}E^{\perp}(\tau)\,d\tau \right\|^{2}_{L^{2}}
\end{align*}
where 
$\delta_{\min}= \min_{j\in K_{\downarrow}}\delta_{j}$ and in the second to last inequality made use of the Parseval identity.

Taking square roots,
\begin{equation*}
\left\Vert \sum_{j\in K_{\downarrow}} u_{j}(t)\varphi_{j} \right\Vert_{L^{2}} \leq \left\|\int_{0}^{t}e^{\delta_{\min}\left(\tau-t\right)}E^{\perp}(\tau)\,d\tau \right\|_{L^{2}}\leq \int_{0}^{t}e^{\delta_{\min}\left(\tau-t\right)}\left\|E^{\perp}\right\|_{L^{2}}\,d\tau
\end{equation*}
and hence we can finally make use of our decay assumption on $E^{\perp}$ to get
\begin{equation*}
\left\Vert \sum_{j\in K_{\downarrow}} u_{j}(t)\varphi_{j} \right\Vert_{L^{2}} \leq \Vert E^{\perp} \Vert_{q}  \int_{0}^{t}  e^{\delta_{\min}(\tau-t)} (T+\tau)^{-q}d\tau.
\end{equation*}
We bound the integral as follows
\begin{align*}
 \int_{0}^{t}  e^{\delta_{\min}(\tau-t)} (T+\tau)^{-q}d\tau & =  \int_{0}^{t/2}  e^{\delta_{\min}(\tau-t)} (T+\tau)^{-q}d\tau +  \int_{t/2}^{t}  e^{\delta_{\min}(\tau-t)} (T+\tau)^{-q}d\tau \\
 & \leq T^{-q} \int_{0}^{t/2}  e^{\delta_{\min}(\tau-t)} d\tau +  (T+t/2)^{-q} \int_{t/2}^{t}  e^{\delta_{\min}(\tau-t)} d\tau \\
  & \leq\delta_{\min}^{-1} T^{-q} \left(  e^{-\delta_{\min}t/2} - e^{-\delta_{\min}t} \right) + \delta_{\min}^{-1} (T+t/2)^{-q} \left( 1- e^{-\delta_{\min}t/2} \right).
\end{align*}
From this we see that
\begin{equation*}
\left\Vert \sum_{j\in K_{\downarrow}} u_{j}(t)\varphi_{j} \right\Vert_{L^{2}} \leq C\Vert E^{\perp}\Vert_{L^{2}_{q}}(T+t)^{-q} 
\end{equation*}
A similar argument holds for the $K_{\uparrow}$ terms. From this, the asserted bounds for $\Vert u\Vert_{L^{2}_{q}}$ follow readily. 

The rest of the proof is devoted to showing that the $C^{2,\alpha}_{q}$ bounds follow from the $L^{2}_{q}$ bounds. By interior parabolic Schauder estimates \cite[Theorem 4.9]{Lieberman:parPDE}, we have that for $t\geq 1$,
\begin{equation*}
\Vert u (t) \Vert_{C^{2,\alpha}}\leq C \left( \sup_{s \in (t-1,t+1)\times M} |u(s,x)| + \Vert E^{\perp}\Vert_{C^{0,\alpha}((t-1,t+1)\times M)} \right).
\end{equation*}
We emphasize that the $C^{2,\alpha}$ norm on the left hand side is the space-time norm on $(t,t+1)\times M$, as defined in \eqref{eq:hold-norm}.

We claim that for $\epsilon >0$, there exists $c(\epsilon)>0$ so that for any function $\varphi \in C^{0,\alpha}(M)$, 
\begin{equation*}
\sup_{x\in M} |\varphi(x)| \leq c(\epsilon) \Vert \varphi\Vert_{L^{2}(M)} + \epsilon \Vert \varphi \Vert_{C^{0,\alpha}(M)}.
\end{equation*}
This follows immediately from an argument by contradiction in conjunction with Arzel\`a--Ascoli. Using this in the Schauder estimate yields (bounding the supremum of the spatial $C^{0,\alpha}(M)$ norm over $t\in(t-1,t+1)$ by the space-time H\"older norm on $(t-1,t+1)\times M$)
\begin{equation*}
\Vert u(t) \Vert_{C^{2,\alpha}}\leq C \left( \sup_{s \in (t-1,t+1)} \Vert u(s,x)\Vert_{L^{2}(M)} + \Vert E^{\perp}\Vert_{C^{0,\alpha}((t-1,t+1)\times M)} \right) + C \epsilon \Vert u(t)\Vert_{C^{0,\alpha}((t-1,t+1)\times M)}.
\end{equation*}
Multiplying by $(T+t)^{q}$ and taking the supremum over $t\geq 1$ yields
\begin{align*}
\sup_{t\geq 1} & \left[ (T+t)^{q}\Vert u(t) \Vert_{C^{2,\alpha}} \right] \\
& \leq C \left( \sup_{t\geq 0} \left[(T+t)^{q} \Vert u(s,x)\Vert_{L^{2}(M)}\right] + \sup_{t\geq 0}\left[ (T+t)^{q}\Vert E^{\perp}\Vert_{C^{0,\alpha}((t,t+1)\times M)}\right] \right)\\
& \qquad +  C\epsilon\sup_{t\geq 0}\left[ (T+t)^{q} \Vert u(t)\Vert_{C^{0,\alpha}((t-1,t+1)\times M)}\right]\\
& = C\left(\Vert u\Vert_{L^{2}_{q}} + \Vert E^{\perp}\Vert_{C^{0,\alpha}_{q}}\right) +C  \epsilon  \Vert u \Vert_{C^{0,\alpha}_{q}}\\
& \leq C\left(\Vert E^{\perp} \Vert_{L^{2}_{q}} + \Vert E^{\perp}\Vert_{C^{0,\alpha}_{q}}\right)+C  \epsilon  \Vert u \Vert_{C^{0,\alpha}_{q}}\\
& \leq C \Vert E^{\perp}\Vert_{C^{0,\alpha}_{q}}+C  \epsilon  \Vert u \Vert_{C^{0,\alpha}_{q}}.
\end{align*}
To finish the proof, it remains to extend the supremum up to $t=0$, because then we may absorb the second term back into the left hand side of the inequality by choosing $\epsilon$ sufficiently small. This may be achieved via global Schauder estimates \cite[Theorem 4.28]{Lieberman:parPDE}
\begin{align*}
\Vert & u(t)\Vert_{C^{2,\alpha}((0,1)\times M)}\\
& \leq C \left( \sup_{s \in (0,1)} \Vert u(s,x)\Vert_{L^{2}(M)} +\epsilon \Vert u\Vert_{C^{0,\alpha}((0,1)\times M)} + \Vert E^{\perp}\Vert_{C^{0,\alpha}((0,2)\times M)}  + \Vert u(0)\Vert_{C^{2,\alpha}(M)}\right).
\end{align*}
Note that 
\begin{equation*}
u(0) = - \sum_{j\in K_{\uparrow}} \left(\int_{0}^{\infty}e^{\delta_{j}\tau} E_{j}^{\perp}(\tau) d\tau\right) \varphi_{j}.
\end{equation*}
The space $\Lambda_{\uparrow}$ is finite dimensional, so there must be a uniform constant $C>0$ so that $\Vert \varphi_{j}\Vert_{C^{2,\alpha}(M)}\leq C \Vert \varphi_{j}\Vert_{L^{2}(M)}$ for all $j\in K_{\uparrow}$. Using this we have that
\begin{align*}
\Vert u(0)\Vert_{C^{2,\alpha}(M)}^{2} & \leq C\sum_{j\in K_{\uparrow}} \left(\int_{0}^{\infty}e^{\delta_{j}\tau} E_{j}^{\perp}(\tau) d\tau\right)^{2} \Vert \varphi_{j}\Vert_{C^{2,\alpha}(M)}^{2}\\
&  \leq C\sum_{j\in K_{\uparrow}} \left(\int_{0}^{\infty}e^{\delta_{j}\tau} E_{j}^{\perp}(\tau) d\tau\right)^{2} \Vert \varphi_{j}\Vert_{L^{2}(M)}^{2}\\
& =C \Vert u(0)\Vert_{L^{2}(M)}^{2}.
\end{align*}
Using the $L^{2}_{q}$ bound obtained above, we may extend the supremum to $t\geq 0$, and conclude the desired H\"older bounds (absorbing the $C^{0,\alpha}$ norms of $u$ into the left hand side, by choosing $\epsilon$ small). 
\end{proof}

\subsection{Construction of a slowly converging flow}
\label{ConstructionSubSec}

To proceed further, we define the norm
\begin{equation*}
\Vert f\Vert _{\gamma}^{*} : = \Vert  \proj_{\Lambda_{0}} f \Vert _{C^{0,\alpha}_{1,\gamma}} + \Vert \proj_{\Lambda_{0}^{\perp}} f\Vert _{C^{2,\alpha}_{1+\gamma}}.
\end{equation*}
Recall that 
\begin{equation*}
\Vert  u \Vert _{C^{0,\alpha}_{1,\gamma}} = \sup_{t\geq 0}\left[ (T+t)^{\gamma} \Vert u(t)\Vert_{C^{0,\alpha}}\right] + \sup_{t\geq 0} \left[ (T+t)^{1+\gamma}  \Vert u'(t)\Vert_{C^{0,\alpha}}\right],
\end{equation*}
and
\begin{equation*}
\Vert u \Vert _{C^{2,\alpha}_{1+\gamma}} = \sup_{t\geq 0} \left[ (T+t)^{1+\gamma} \Vert u(t)\Vert_{C^{2,\alpha}} \right].
\end{equation*}
We emphasize that these H\"older norms are the space-time H\"older norms, defined in \eqref{eq:hold-norm}. For $\gamma$ to be specified below, we define $X$ to be the Banach space of functions $f$ with $\Vert f\Vert_{\gamma}^{*}<\infty$.
\begin{prop}
\label{prop:contract-map}
Assume that $g_{\infty}$ satisfies $AS_{p}$. We may thus fix a point where $F_{p}|_{\SS^{k-1}}$ achieves a positive maximum and denote it by $\hat v$. Define
\begin{equation*}
\varphi (t) = (T+t)^{-\frac{1}{p-2}} \left( \frac{2(N-2)}{p(p-2)F_{p}(\hat v) } \right)^{\frac{1}{p-2}} \hat v,
\end{equation*}
as in Lemma \ref{lemm:poly-decay-def-varphi}. Then, there exists $C>0$, $T>0$, $\frac{1}{p-2} < \gamma < \frac{2}{p-2}$ and $u(t) \in C^{\infty}(M \times (0,\infty))$ so that $u(t) > 0$ for all $t >0$, $g(t) : = u(t)^{N-2}g_{\infty}$ is a solution to the Yamabe Flow, and
\begin{equation*}
\Vert w^{\T}(t) + \Phi(\varphi(t)+w^{\T}(t)) + w^{\perp} (t) \Vert^{*}_{\gamma} = \Vert u(t) - \varphi(t) - 1\Vert_{\gamma}^{*} \leq C.
\end{equation*}
\end{prop}
\begin{proof}
We fix $\frac{1}{p-2} < \gamma < \frac{2}{p-2}$ so that $\gamma \not\in\{\frac{\mu_{1}}{2(N-2)},\ldots,\frac{\mu_{k}}{2(N-2)}\}$. By Proposition \ref{prop:slow-decay-rewrite-two-comp}, it is enough to solve
\begin{equation*}
\begin{split}
2(N-2)(w^{\T})' +D^{2} F_{p}(\varphi) w^{\T} & = E^{\T}(w)\\
(w^{\perp}) ' - \cL_{\infty}w^{\perp} & = E^{\perp}(w),
\end{split}
\end{equation*}
for $w(t)$ with $\Vert w\Vert _{\gamma}^{*}  < C$. To do so, we will use the contraction mapping method. We define a map 
\begin{equation*}
S:\{ w \in X : \Vert w\Vert _{\gamma}^{*} \leq 1\} \to X =\{w :  \Vert w\Vert _{\gamma}^{*} < \infty\},
\end{equation*}
by defining $u := \proj_{\Lambda_{0}} S(w)$ to be the solution of 
\begin{equation*}
2(N-2)u' +D^{2} F_{p}(\varphi) u  = E^{\T}(w),
\end{equation*} and $ v := \proj_{\Lambda_{0}^{\perp}} S(w)$ to be the solution of 
\begin{equation*}
v ' - \cL_{\infty}v  = E^{\perp}(w).
\end{equation*}
From this, we have defined the map $S(w)$ by its orthogonal projection onto $\Lambda_{0}$ and $\Lambda_{0}^{\perp}$. These solutions exist, in the right function
spaces, by combining the bounds for the error terms in Proposition \ref{prop:slow-decay-rewrite-two-comp} with Lemmas \ref{lemm:ker-proj-ODE} and \ref{lemm:ker-orthog-prob}. Furthermore, we have the explicit bound 
\begin{align*}
\Vert \proj_{\Lambda_{0}}S(w)\Vert _{C^{0,\alpha}_{1,\gamma}} &  \leq c \Vert E^{\T}(w)\Vert _{C^{0,\alpha}_{1+\gamma}}\\
& \leq c \sup_{t\geq 0} (T+t)^{1+\gamma} ( (T+t)^{-1-\frac {1}{p-2}}+\Vert w^{\T} \Vert_{C^{0,\alpha}}^{p-1} + \Vert w^{\perp}\Vert_{C^{2,\alpha}} ) ( (T+t)^{-\frac {1}{p-2}}+ \Vert w \Vert_{C^{2,\alpha}} ) \\
& \qquad + c \sup_{t\geq 0}((T+t)^{\gamma-\frac{2}{p-2}} +  (T+t)^{\gamma-\frac{1}{p-2}} \Vert w^{\T}\Vert _{C^{2,\alpha}})\\
& \qquad   + c\sup_{t\geq0}((T+t)^{\gamma+\frac{1}{p-2}}\Vert w^{\T}\Vert ^{2}_{C^{2,\alpha}}+ (T+t)^{1+\gamma} \Vert w^{\T}\Vert _{C^{2,\alpha}}^{p-1} ) \\
& \qquad + c \sup_{t\geq0} (T+t)^{1+\gamma}  \Vert w^{\perp} \Vert_{C^{2,\alpha}}^{2}\\
& \leq c\left(T^{\gamma-\frac{2}{p-2}} + \left(T^{-\frac{1}{p-2}} + 
T^{(p-2)\left( \frac{1}{p-2} - \gamma\right)}\right) \Vert w\Vert _{\gamma}^{*}\right).
\end{align*}
Here, we have absorbed powers of $(T+t)$ into the various $w$ norms, and bounded this by $\Vert w\Vert_{\gamma}^{*}$. Note that the $w^{\perp}$ terms in $\Vert w\Vert_{\gamma}^{*}$ are multiplied by $(T+t)^{1+\gamma}$, but the $w^{\T}$ term in $\Vert w\Vert_{\gamma}^{*}$ is only multiplied by $(T+t)^{\gamma}$, so we cannot absorb as high of a power of $(T+t)$ into it (fortunately, the $w^{\T}$ terms are all raised to a large power, or already multiplied by an appropriately decaying power of $(T+t)$, as is easily checked). In the last step, we have used the bound
\begin{equation*}
\Vert w^{\T}\Vert_{C^{2,\alpha}((t,t+1)\times M)} \leq c\left( \Vert w^{\T}\Vert_{C^{0,\alpha}((t,t+1)\times M)} + \Vert (w^{\T})'\Vert_{C^{0,\alpha}((t,t+1)\times M)} \right).
\end{equation*}
This is a consequence of the fact that $\Lambda_{0}$ is finite dimensional (so any two norms on it are uniformly equivalent) and that the parabolic\footnote{We emphasize that the \emph{space-time} $C^{k,\alpha}$-norms on $\Lambda_{0}$ are \emph{not} all uniformly equivalent. This is due to the fact that (as usual) the time dependence of the functions turns the space into an infinite dimensional vector space. In the asserted inequality, we have used the fact that the \emph{spatial} $C^{k,\alpha}$-norms of any element in $\Lambda_{0}$ are all equivalent to any other $C^{k',\alpha'}$. The asserted inequality follows from this, along with the fact that in the space-time $C^{2,\alpha}$-norm, there is at most one single time derivative (which does not come with any spatial derivatives).} $C^{2,\alpha}$ H\"older norms only contain at most one time derivative (which does not come paired with any spatial derivatives). Similarly,
\begin{align*}
 \Vert \proj_{\Lambda_{0}^{\perp}} S(w) \Vert _{C^{2,\alpha}_{1+\gamma}}  & \leq \Vert E^{\perp}(w)\Vert _{C^{0,\alpha}_{1+\gamma}}\\
& \leq c \sup_{t\geq 0} (T+t)^{1+\gamma} ( (T+t)^{-1-\frac {1}{p-2}}+\Vert w^{\T} \Vert_{C^{0,\alpha}}^{p-1} + \Vert w^{\perp}\Vert_{C^{2,\alpha}} ) ( (T+t)^{-\frac {1}{p-2}}+ \Vert w \Vert_{C^{2,\alpha}} ) \\
& \qquad + c\sup_{t\geq0} \left[((T+t)^{1+\gamma-\frac{1}{p-2}} + (T+t)^{1+\gamma} \Vert w\Vert _{C^{2,\alpha}}) \Vert w^{\perp}\Vert _{C^{2,\alpha}} \right]\\
& \qquad + c\sup_{t\geq0}\left[((T+t)^{1+\gamma-\frac{1}{p-2}} + (T+t)^{1+\gamma} \Vert w\Vert _{C^{2,\alpha}})((T+t)^{-1-\frac{1}{p-2}} + \Vert w'\Vert _{C^{0,\alpha}})\right]\\
& \leq c\left(T^{\gamma-\frac{2}{p-2}} + \left(T^{-\frac{1}{p-2}} + T^{(p-2)\left( \frac{1}{p-2} - \gamma\right)}\right) \Vert w\Vert _{\gamma}^{*}\right).
\end{align*}
Thus, because $\gamma \in \left( \frac{1}{p-2},\frac{2}{p-2}\right)$, by choosing $T$ large enough, we can ensure that $S$ maps $\{ w : \Vert w\Vert _{\gamma}^{*} \leq 1\} \subset X$ into itself. Finally, we check that we can guarantee that $S$ is a contraction mapping by taking $T$ even larger if necessary. The following inequalities are proven by the same argument we have just used
\begin{align*}
\Vert \proj_{\Lambda_{0}}S(w_{1}) - \proj_{\Lambda_{0}}S(w_{2})\Vert _{C^{0}_{1,\gamma}} &\leq c \left(T^{-\frac{1}{p-2}} + T^{(p-2)\left( \frac{1}{p-2} - \gamma\right)}\right) \Vert w_{1}-w_{2}\Vert _{\gamma}^{*}\\
\Vert \proj_{\Lambda_{0}^{\perp}}S(w_{1}) - \proj_{\Lambda_{0}^{\perp}}S(w_{2})\Vert _{C^{2,\alpha}_{1+\gamma}} &\leq c \left(T^{-\frac{1}{p-2}} + T^{(p-2)\left( \frac{1}{p-2} - \gamma\right)}\right)  \Vert w_{1}-w_{2}\Vert _{\gamma}^{*}.
\end{align*}
Thus, by enlarging $T$ if necessary, we have that $S$ is a contraction map. This finishes the proof. 
\end{proof}

We now show how the previous proposition yields solutions converging at exactly a polynomial rate. 

\begin{proof}[Proof of Theorem \ref{theo:gen-slow-conv}]
From Propostion \ref{prop:slow-decay-rewrite-two-comp}, we have constructed $\varphi(t)$ and $u(t)$ so that \begin{equation*}
\varphi (t) = (T+t)^{-\frac{1}{p-2}} \left( \frac{2(N-2)}{p(p-2)F_{p}(\hat v) } \right)^{\frac{1}{p-2}} \hat v,
\end{equation*}
$u(t)^{N-2}g_{\infty}$ is a solution to the Yamabe flow, and
\begin{equation*}
u(t) = 1+\varphi(t) + \tilde w(t) : = 1+\varphi(t) +w^{\T}(t) + \Phi(\varphi(t)+w^{\T}(t)) + w^{\perp} (t),
\end{equation*}
where $\tilde w(t)$ satisfies (in particular) $\Vert \tilde w\Vert_{C^{0}} \leq C (1+t)^{-\gamma}$ for some $C>0$ and all $t\geq 0$. We have arranged that $\gamma > \frac{1}{p-2}$, which implies that $\varphi(t)$ is decaying slower than $\tilde w(t)$.
Thus
\begin{equation*}
\Vert u(t) - 1\Vert_{C^{0}}\geq C(1+t)^{-\frac{1}{p-2}}
\end{equation*}
as $t\to\infty$. From this, the assertion follows. 
\end{proof}

\section{Examples satisfying $AS_{p}$} \label{sec:examp}

In this section we provide explicit examples of metrics which satisfy $AS_{p}$ for both $p=3$ and $p\geq 4$. This allows us, via Theorem \ref{theo:gen-slow-conv}, to conclude the existence of slowly converging Yamabe flows. 

\subsection{Examples which satisfy $AS_{3}$} In this subsection we prove Proposition \ref{prop:CPn-example}. Suppose that we are given integers $n,m > 1$ and a closed $m$-dimensional Riemannian manifold $(M^{m},g_{M})$ with constant scalar curvature $R_{g_{M}} \equiv 4(n+1)(m+n-1)$. We denote the complex projective space equipped with the Fubini--Study metric (our normalization of the Fubini--Study metric is as follows: we define $\CC P^{n}$ to be the metric induced by the Riemannian submersion from the unit radius sphere $\SS^{2n+1}\to \CC P^{n}$) by $(\CC P^{n},g_{FS})$. We will show that the product metric $(M^{m}\times \CC P^{n}, g_{M} \oplus g_{FS})$ is a degenerate critical point satisfying $AS_{3}$. Recall that this implies that the metric is non-integrable, by Lemma \ref{lemm:int-imp-F-zero}.

Write $g:=  g_{M} \oplus g_{FS}$. Because $R_{g_{FS}} = 4n(n+1)$ \cite[p.\ 86]{Petersen} it follows that the scalar curvature of $g$ satisfies
$R_{g} = 4(n+1)(m+n-1)+4n(n+1) = 4(n+1)(m+2n-1).$
The dimension of $M^{m}\times \CC P^{n}$ is $m +2n$, so $\Lambda_{0}$ consists of eigenfunctions of $\Delta_{g}$ with eigenvalue $\frac{R_{g}}{m+2n-1} = 4(n+1)$. Because $\lambda_{1}(g_{FS}) = 4(n+1)$ \cite[Proposition C.III.1]{BGM}, we see that $(M^{m}\times \CC P^{n},g)$ is degenerate; for any first eigenfunction $v$ on $\CC P^{n}$, the function $1\otimes v$ on $M^{m}\times \CC P^{n}$ will be an eigenfunction of $\Delta_{g}$ with eigenvalue $4(n+1)$. 

The eigenfunctions of $\Delta_{g_{FS}}$ may be explicitly constructed by considering polynomials on $\CC^{n}$ which are homogeneous of degree $k$ in both $z$ and $\overline z$ and which are harmonic. These polynomials restrict to $\SS^{2n+1}$ and are invariant under the natural $\SS^{1}$ action, so they thus descend to the quotient. This is described in detail in \cite[Proposition C.III.1]{BGM}. By a recent observation of Kr\"oncke, \cite[p.\ 25]{Kroncke:2013fk}, the harmonic polynomial
$
h(z,\overline z) : = z_{1}\overline z_{2} + z_{2}\overline z_{1} +z_{2}\overline z_{3} + z_{3}\overline z_{2}  + z_{3}\overline z_{1} + z_{1}\overline z_{3},$
defined on $\CC^{n+1}$ for $n\geq 2$, descends to a first eigenfunction $v$ of $\Delta_{g_{FS}}$ for which $\int_{\CC P^{n}} v^{3} dV_{g_{FS}}\not = 0$. 
The function $1\otimes v$ is an eigenfunction of $\Delta_{g}$ with eigenvalue $4(n+1)$, so it is an element of $\Lambda_{0}$. Moreover, by Fubini's theorem,
$\int_{M^{m}\times \CC P^{n}} (1\otimes v)^{3} dV_{g} = \vol(M^{m},g_{M}) \int_{\CC P^{n}} v^{3} dV_{g_{FS}} \not = 0.
$ 
Thus, we see that $(M^{m}\times \CC P^{n},g)$ is degenerate and by \eqref{eq:F3-comp}, the function $F_{3}$ is not everywhere zero on $\Lambda_{0}$. This shows that $(M^{m}\times \CC P^{n},g)$ satisfies $AS_{3}$, as claimed.

\subsection{Examples satisfying $AS_{p}$ for $p\geq 4$}
This subsection is devoted to the detailed study of the Yamabe problem on $\SS^{1}(R)\times\SS^{n-1}$. 
Our goal is to obtain examples in any dimension $n\ge3$ of a non-integrable critical point of $\cY$ which satisfies the condition $AS_{p}$ for some $p\geq 4$, as defined in Definition \ref{defi:ASp}. 
The study involves properties of a certain \emph{period function} $\tau(\alpha)$ defined below. 
Here, we start by giving an overview of Schoen's discussion \cite{Schoen:Montecatini}, supplying 
detailed proofs.
The main new observation is that these facts imply the existence of a constant
scalar curvature metric satisfying the assumptions of Theorem \ref{theo:gen-slow-conv}.
We observe that the same ODE which we analyze has been considered, from a different perspective, in \cite{MPU} where the authors analyze moduli spaces of singular Yamabe metrics.

\subsubsection{An ODE parametrizing all solutions of the Yamabe problem}

We consider a one-parameter family of conformal classes $[g_{T}]$ on $\SS^{1}\times \SS^{n-1}$ represented by the natural product metric  $\SS^{1}(T/2\pi)\times \SS^{n-1}(1)$ (here $\SS^{k}(r)$ is the $k$-sphere of radius $r$ in $\RR^{k+1}$). We will write $t$ for the coordinate on $\SS^{1}(T/2\pi)$.

\begin{prop}[\cite{Schoen:Montecatini}]\label{prop:schoen-analysis}
Let $u_0=\left( \frac{n-2}{n}\right)^{\frac{n-2}4} = \left(\frac 2 N\right)^{\frac{1}{N-2}}$. Then, exists a map $\tau: (u_{0},1)\to \RR_{>0}$ which parametrizes solutions to the Yamabe problem on $\SS^{1}\times \SS^{n-1}$ in the following sense: For a given $T>0$, up to scaling the conformal factor, the complete list of constant scalar curvature metrics in $[g_{T}]$ is (1) the product metric and (2) a metric of the form $u(t)^{N-2}g_{T}$ where $u(t)$ solves the ODE
\begin{equation*}
4u''-(n-2)^{2}u + {n(n-2)} u^{\frac{n+2}{n-2}} = 0,
\end{equation*}
with initial conditions $(u(t_{0}),u'(t_{0})) = (\alpha,0)$ for some $t_{0} \in \SS^{1}(T/2\pi)$. Here, $\alpha \in (u_{0},1)$ is any solution of $\tau(\alpha) = \frac{T}{k} $ with $k$ an arbitrary positive integer.
\end{prop}
\begin{proof}
We will follow Schoen and look for solutions to the Yamabe problem with constant scalar curvature $n(n-1)$ (equal to that of the unit sphere), and in doing so we drop the volume constraint. A crucial observation is that by a result of Caffarelli--Gidas--Spruck (following the classical work of Gidas--Ni--Nirenberg), a constant scalar curvature metric in $[g_{T}]$ must have conformal class only depending on the $\SS^{1}$-variable $t$ (see \cite{GidasNiNirenberg,CaffarelliGidasSpruck}). As such, this reduces the problem to studying an ODE rather than a PDE. 

It will be convenient to lift the analysis to the universal cover $\RR\times \SS^{n-1}$ and use $(t,\xi)\in\RR\times\RR^n$
with $|\xi|=1$ as coordinates. 
In particular, we will forget about $g_{T}$ for now, and consider instead the metric 
$g=dt^{2}+ g_{\SS^{n-1}(1)}
$ 
on $\RR\times \SS^{n-1}$. 
Then, a solution to the Yamabe Problem in $[g_{T}]$ will correspond to a function $u(t)$ on $\RR\times \SS^{n-1}$ (depending only on the first factor) with period $T$ in $t$ for which $u(t)^{\frac{4}{n-2}}g$ has constant scalar curvature $n(n-1)$. 

Now, $u(t)^{\frac{4}{n-2}}g$ having constant scalar curvature $n(n-1)$ is equivalent to the ODE 
\begin{equation}\label{eq:YP-t}
4u''-(n-2)^{2}u + {n(n-2)} u^{\frac{n+2}{n-2}} = 0,
\end{equation}
as $R_{g} = R_{g_{\SS^{n-1}(1)}} = (n-1)(n-2)$ and
$n(n-1)=R_{u^{N-2}g} = -(N+2) u^{-\frac{n+2}{n-2}} \left(u''- \frac{1}{N+2} R_{g}u \right).
$ 
There is one obvious solution to \eqref{eq:YP-t} given by the constant $u(t) \equiv u_{0} = \left( \frac{n-2}{n}\right)^{\frac{n-2}4}$.
This simply corresponds to the rescaling of $g_{T}$ so that it has scalar curvature $n(n-1)$, as desired. 

There is a second explicit solution to \eqref{eq:YP-t} obtained by considering $\RR\times \SS^{n-1}$ as the coordinate patch 
of $\SS^{n}$ given by $\SS^{n}-\{N,S\}$, the sphere minus two antipodal points. 
The restriction of the standard metric on $\SS^{n}$ to $\SS^{n}-\{N,S\}$
(which has scalar curvature $n(n-1)$)
then produces a solution to  \eqref{eq:YP-t}
as long as we can check that this metric is conformally related to $g$. To see this, notice that the map $\Psi: (\RR\times \SS^{n-1} ,g)\to 
(\RR^{n}-\{0\},g_{\Euc})$, $(t,\xi) \mapsto e^{t}\xi$ is conformal, because
\begin{equation*}
\Psi^{*} g_{\Euc} = \Psi^{*}\left( dr^{2} + r^{2} g_{\SS^{n-1}}\right) = e^{2t} dt^{2} + e^{2t} g_{\SS^{n-1}} = e^{2t} g,
\end{equation*}
where $g_{\Euc}$ denotes the Euclidean metric on $\RR^{n}\setminus\{0\}$).
On the other hand, by stereographic projection the spherical coordinate patch on $\RR^{n}-\{0\}$ has the metric
\begin{equation*}
g_{\SS^{n}} = \frac{4 g_{\Euc}}{(1+r^{2})^{2}}, \quad
\hbox{where\ }  r=\left|x\right|, 
\quad \hbox{with $x\in\RR^n\setminus\{0\}$}.
\end{equation*}
Thus
$\Psi^{*}(g_{\SS^{n}}) = \frac{4e^{2t}}{(1+e^{2t})^{2}} g = (\cosh t)^{-2} g.$
Therefore, we have another solution to \eqref{eq:YP-t} given by 
$u_{1}(t) = (\cosh t)^{-(n-2)/2}.$
Of course, the metric $u_{1}(t)^{{4}/{(n-2)}} g$ does not descend to the quotient $\SS^{1}(T/2\pi)\times \SS^{n-1}$ 
(and it is not even a complete metric on $\RR\times \SS^{n-1}$) but it will prove useful in the sequel.

By setting $v = \frac{du}{dt}$,
\eqref{eq:YP-t} can be converted to a first order system 
\begin{equation}
\label{eq:YP-system}
\frac{d}{dt} (u,v) = X(u,v),
\end{equation}
where the vector field $X$ on the 
$uv$-plane is defined by 
$X(u,v) = \left( v ,  \frac{{(n-2)^{2}} u - {n(n-2)}u^{\frac{n+2}{n-2}}}{4}\right).$
We note that the second component of $X$ is negative when $u > u_{0}$ and positive when $u< u_{0}$. From the above analysis, we know that 
\begin{equation}
(u(t),v(t)) = (u_{1}(t), u_{1}'(t)) = \left( (\cosh t)^{-\frac{n-2}{2}}, \left(\frac{1}{4}-\frac n8 \right) \frac{\sinh t}{(\cosh t)^{\frac n 2}}\right)
\end{equation}
is a solution to \eqref{eq:YP-system}. Letting $t=0$, we see that $(1,0)$ is on this integral curve. Additionally, 
letting $t \to \pm \infty$ 
(note here
that $n\geq 3$), we see that the curve tends to $(0,0)$. Thus, the orbit associated to $u_{1}$, along with $(0,0)$ encloses a region $\Omega$ with compact closure $\overline\Omega$ and such $\Omega$ is invariant under the flow since its boundary is a homoclinic cycle (i.e., a trajectory that limits to the same critical point at $t=\pm\infty$).

\begin{claim}
Any periodic solution with $u > 0$ for all time must lie inside 
$\Omega$. 
\end{claim}

\begin{proof}
By the previous comments, it suffices to consider a trajectory $\gamma(t)=(u(t),v(t))$ in $\RR^2\setminus\Omega$.
Observe that 
\begin{equation}
\label{vdotEq}
4v'(t)=n(n-2)u(u_0^{N-2}-u^{N-2});
\end{equation}
thus, whenever $u(t)>u_0$ then $v'(t)<0$.
We divide the proof into two cases.

{\it Case 1: $u(0)>u_0$.}
In this case, since $\gamma(t)$ is defined globally for $t\in\RR$ and is periodic we claim that there exists $O_0$ such that $u(O_0)=u_0$, and $v'(t)<0$ for $t\in(0,O_0)$. Indeed, if it were not the case monotonicity of the second component of $\gamma$ should immediately imply that the corresponding trajectory of $\gamma(\cdot)$ be unbounded, contradicting the periodicity assumption.

As a result, clearly $v(O_{0})<0$ hence the system implies $u'(O_0)<0$ and thus it follows that $v(t)>v(O_0)$ for $t>O_0$. Since $\gamma$ is global, must remain
in the right half space (as $u>0$) and by monotonicity of the second component it follows that $\gamma$ must approach $(0,0)$ as 
$t\nearrow\infty$. But then it is not periodic. (Note that in fact, by time-reversal symmetry, $\gamma$ should be a homoclinic cycle.)

{\it Case 2: $u(0)<u_0$.} We can assume $v(0)>0$ since otherwise we reduce to the last part of the proof of Case 1. Thus $v(0)>0$, therefore $u'(0)>0$ and $v'(0)>0$. It follows
once again that $\gamma$ crosses the vertical line $u=u_0$, and then we are in Case 1.
\end{proof}

\let\al=\alpha
\def\Im{\hbox{Im}\,}

\begin{claim}
Suppose that $\gamma_\al(t)=(u(t),v(t))$ solves \eqref{eq:YP-system}, and that 
$\gamma_\al(0)=(\alpha,0)\in\Omega$. Then either 
$\gamma_\al\equiv(u_0,0)$ or $\gamma_\al$ is a smooth periodic orbit
contained in $\Omega-\partial\Omega$. 
Furthermore, if $u_0\le\alpha_1<\alpha_2$, then $\gamma_{\al_1}$ is enclosed by $\gamma_{\al_2}$.
\end{claim}

\begin{proof}
For the first statement, let us start by observing that (by compactness of $\overline{\Omega}$) any such solution $\gamma_{\al}(t)$ must be defined for all times and $\gamma_{\al}(t)=\gamma_{\al}(-t)$ for all $t\in\RR$ by means of a standard ODE uniqueness argument (as $X(u,v)=X(u,-v)$).  
To proceed further, let us recall that the flow \eqref{eq:YP-system} is generated by the Hamiltonian (see \ \cite[(2.3)]{MPU})
\begin{equation}\label{HamEq}
H(u,v)=2{v^2}+\frac{(n-2)^2(u^N-u^2)}{2}.
\end{equation}
The corresponding conservation law (together with the fact that $H(u,0)=0$ implies $u\in\{0,1\}$) rules out the existence of solutions $\gamma_{\al}$ such that $\lim_{t\to +\infty}\gamma_{\al}(t)=(0,0)$ (and thus $\lim_{t\to -\infty}\gamma_{\al}(t)=(0,0)$ as well) or $\lim_{t\to +\infty}\gamma_{\al}(t)=(1,0)$ (and thus $\lim_{t\to -\infty}\gamma_{\al}(t)=(1,0)$) whenever $\alpha\in(0,1)$. 
Then the first claim follows from the fact that $\gamma_\alpha$ for $\alpha\in(0,1)$ must intersect the $u$-axis (exactly) twice, and away from the zeros of the vector field $X$. 
Then by the time-reversal symmetry we conclude that $\gamma_\al$ must be periodic, hence also smooth. These arguments show in particular that $\Im\gamma_\al$ for such $\al$ is diffeomorphic to $\SS^1$. Uniqueness of solutions to ODEs then implies the last claim.
\end{proof}

Finally, this allows us to conclude the general classification of constant scalar curvature metrics in $[g_{T}]$. Any constant curvature metric must depend only on the $t$ variable and thus lift to a solution $u(t)$ to the ODE \eqref{eq:YP-t} on $\RR$. The solution must periodic of period $\frac{T}{k}$ for some positive integer $k$, as the metric descends to $\SS^{1}(T/2\pi) \times \SS^{n-1}$. Thus, by the above claim, there exists $\alpha \in (u_{0},1)$ so that $u(t)$ solves the ODE with initial conditions $(\alpha,0)$ (after possibly shifting $u(t)$ by a rotation of $\SS^{1}$). Of course, if $k>1$, then what we mean here is that the conformal factor on $\SS^{1}(T/2\pi)\times \SS^{n-1}$ is $u(t)$ concatenated $k$ times. By definition $\tau(\alpha) = \frac{T}{k}$. This completes the proof of the claim.
\end{proof}

\subsubsection{The period function}
\label{PeriodSubsec}

\begin{lemma}[\cite{Schoen:Montecatini}]
\label{LimitAtInftyLemma}
The period function $\tau(\alpha)$ is continuous on the interval $(u_{0},1)$. Furthermore, it satisfies
(i) $\lim_{\alpha\nearrow 1} \tau(\alpha) = +\infty$ and (ii) $\lim_{\alpha\searrow u_{0}} \tau(\alpha) = (n-2)^{-1/2}2\pi := T_{0}$.
\end{lemma}
\begin{proof}
(i) Suppose that there is a sequence $\alpha_{k} \nearrow 1$ so that $\tau(\alpha_{k}) \leq C$ for some constant $C$. By possibly extracting a subsequence, we may assume that $\lim_{k\to\infty}\tau(\alpha_{k}) = T_{\infty} < \infty$. Now, consider the points $\gamma_{\alpha_{k}}(T_{\infty}/2)$. By making use of the equation $H(\alpha_{k},0) = H(u_{\alpha_{k}}(\tau(\alpha_{k})/2),0)$, we now claim that $u_{\alpha_{k}}(\tau(\alpha_{k})/2) \to 0$ as $k \to \infty$. Indeed,
$u_{\alpha_{k}}(\tau(\alpha_{k})/2) \in [0,u_{0}]$, so we may assume that it converges to some value $u_{\infty}$ by further extracting a subsequence. Thus, taking the limit as $k\to\infty$ of $H(u_{\alpha_{k}}(\tau(\alpha_{k})/2),0) = H(\alpha_{k},0)$, we get that 
\begin{equation}
u_{\infty}^{2} ( u^{N-2}_{\infty} - 1) = 0.
\end{equation}
However, because $u_{\infty} \leq u_{0}$, the second term must be negative, so $u_{\infty} = 0$. Thus, we see that $\gamma_{\alpha_{k}}(T_{\infty}/2)$ must  converge to $(0,0)$. 
On the other hand, by continuous dependence of solutions to ODEs on their initial data, $\gamma_{\alpha_{k}}(T_{\infty}/2)$ must converge to $\gamma_{1}(T_{\infty}/2)$ which cannot be $(0,0)$. This is a contradiction.

(ii)
We will show this by proving that as $\alpha \searrow u_{0}$, if we rescale the solutions, then they converge to a solution of the linearized ODE around $(u_{0},0)$. We shift $u_{0}$ to the origin and blow up by defining $(\tilde u,\tilde v) = \left(\frac{u-u_{0}}{\alpha-u_{0}},\frac{v}{\alpha-u_{0}}\right)$. Thus the ODE becomes
\begin{equation}
\frac{d}{dt} (\tilde u,\tilde v) = \tilde X_{\alpha} : = \left( \tilde v , 
\frac{n(n-2)}{4} \frac{(\alpha-u_{0})\tilde u + u_{0}}{\alpha-u_{0}}  \left( u_{0}^{N-2} - ((\alpha-u_{0}) \tilde u+u_{0})^{N-2}\right)\right).
\end{equation}
Notice that under the rescaling, the trajectory $\tilde \gamma_{\alpha}$ encircles the origin and contains the point $(1,0)$. Moreover, as $\alpha \searrow u_{0}$ the vector field $\tilde X_{\alpha}$ converges to
\begin{equation}
\tilde X_{u_{0}} =  \left( \tilde v ,  - \frac{n(n-2) (N-2)}{4}u_{0}^{N-2} \tilde u \right) =  \left( \tilde v , (2-n) \tilde u \right).
\end{equation}
Thus, the solution to the linearized equation is $\tilde\gamma_{u_{0}}:= (\cos((n-2)^{1/2} t), -(n-2)^{1/2} \sin((n-2)^{1/2}t)$, 
which is periodic with period given by 
$T_{0}:= (n-2)^{-1/2}2\pi$.

Now, we claim first that the $\tau(\alpha)$ are bounded as $\alpha \to u_{0}$, say $\tau(\alpha) \leq 5T_{0}/2$. Suppose not, so there are $k\to\infty$ and $\alpha_{k}\searrow u_{0}$ so that $\tau(\alpha_{k}) > 5T_{0}/2$. Now, on one hand, by continuous dependence on initial data and due to the explicit formula of $\tilde\gamma_{u_{0}}$ we have that for any fixed $t \in (T_{0}/2,T_{0})$ for $k$ large enough the trajectory $\tilde \gamma_{\alpha_{k}}(t)$ has $\tilde v_{\alpha_{k}}(t) \geq \epsilon >0$. On the other hand (by Claim 19), because we have assumed that $T_{0} < 2\tau\left(\alpha_{k}\right) /5 < \tau\left(\alpha_{k}\right)/2$, $\tilde \gamma_{\alpha_{k}}(t)$ must always have $\tilde v_{\alpha_{k}}(t) < 0$, because $\tau(\alpha_{k})/2$ is the first (positive) time when $\tilde \gamma_{\alpha_{k}}$ crosses the $\tilde{u}$-axis. This is a contradiction.

That being said, because $\tau(\alpha)$ is bounded for $\alpha$ close to $u_{0}$, for any $\alpha_{k}\searrow \alpha$, we may assume that $\tau(\alpha_{k})\to \overline T$ for some $\overline T$. By continuous dependence of ODEs on their parameters, thus $\lim_{k\to\infty}\tilde \gamma_{\alpha_{k}}(\tau(\alpha_{k})/2) = \tilde \gamma_{u_{0}}(\overline T/2)$. Because $\tilde \gamma_{\alpha_{k}}(\tau(\alpha_{k})/2)$ all have $\tilde u \leq 0$ and lie on the $\tilde v = 0$ axis, we thus see that $\lim_{k\to\infty} \tilde \gamma_{\alpha_{k}}(\tau(\alpha_{k})/2) = (-1,0)$ and, at the same time, necessarily $\overline{T}=(2q+1)T_{0}$ for some integer $q\in\mathbb{N}$. But if $q>0$ then it were $\overline{T}\geq 3T_{0}$, contradicting our previous argument which showed that instead $\overline{T}\leq 5T_{0}/2$. Hence $q=0$ so that $\overline{T}=T_{0}$ and this completes the proof of (ii).

Continuity of $\tau(\alpha)$ follows by a similar argument as the one used in (i). 
\end{proof}

\subsubsection{Checking $AS_{p}$ for $p\geq 4$}

Proposition \ref{prop:example} follows from the following 
result.
\begin{prop}
The product metric $g_{\infty}$ on $\SS^1\left(\frac{T_0}{2\pi}\right)\times \SS^{n-1}(1)$ is a degenerate critical point of the Yamabe functional. When $n>2$ it is non-integrable, a global minimum of the Yamabe energy, and satisfies $AS_{p}$ for some even $p\ge 4$.
\end{prop}

\def\bpf{\begin{proof}}
\def\epf{\end{proof}}
\def\la{\lambda}
\def\vp{\varphi}

\begin{proof}
We start by proving degeneracy. Note that $R_{g_{\infty}}=R_{g_{\SS^{n-1}(1)}}=(n-2)(n-1)$ so it suffices to show that $\la_1(g_{\infty})=n-2$.
The eigenvalues of $g_{\infty}$ are the sums of those of each of its factors. Therefore, $\la_1(g_{\infty})=\min\{\la_1(\mathbb{S}^{n-1}(1)), \la_1(\mathbb{S}^1(T_0/2\pi)\}= \min\{n-1,(T_0/2\pi)^{-2}\}=n-2$.

%

Monotonicity of the period function follows from the general result 
\cite[Lemma 1.2]{BidautVeronBouhar} or \cite{CWeissler}. We review
the proof in our special setting in Appendix B as it seems not
to be well-known to experts.
Non-integrability is now immediate since $\vp_1(t):=\sin(\sqrt{n-2}t)$ is an eigenfunction of $\cL_\infty$ while Proposition \ref{prop:schoen-analysis} and the fact that $\tau(\alpha)$ is strictly increasing imply that $g_{\infty}$ is the only critical point of the Yamabe energy: because $\tau(\alpha)$ is strictly increasing and $\lim_{\alpha\searrow u_{0}}\tau(\alpha) = T_{0}$, there cannot be $\alpha \in (u_{0},1)$ and integers $k\geq 1$ so that $\tau(\alpha) = \frac{T_{0}}{k}$. Thus, $\Lambda_{0}$ is one dimensional, but $1$ is the \emph{unique} critical point of $\cY$ in $[g_{\infty}]_{1}$, so $g_{\infty}$ must be non-integrable. Notice that because $1$ is the unique critical point of $\cY$ in $[g_{\infty}]_{1}$, the solution of the Yamabe problem guarantees that it is the global minimum of the Yamabe energy on $[g_{\infty}]_{1}$.

Now, because $g$ is a non-integrable critical point, the function $F(v)$ defined on $\Lambda_{0}$ in Proposition \ref{prop:LSred} is necessarily non-constant. Furthermore, because $g$ is a unique global minimum for the Yamabe problem in its conformal class, we see that $\cY(1) < \cY(w)$ for any $w^{N-2}g_{\infty} \in [g_{\infty}]_{1}$ with $w\not\equiv 1$. In particular, this yields that if $v\not =0$ then necessarily $F(0) < F(v)$. Thus, denoting by $p$ the order of integrability of $g$, it is clear that $F_{p}$ must be everywhere non-negative (if it were not, we could take $v$ small enough so that Taylor's theorem would imply that $F(v) < F(0)$, contradicting the previous argument). From this, it is clear that $p \geq 3$ and in fact has to be even. (We remark that one can directly check $p\neq 3$ because $D^{3}F(0)[\varphi_{1},\varphi_{1},\varphi_{1}] = 0$ by using the explicit form of $F_{3}(v)$ given in \eqref{eq:F3-comp}.)
\end{proof}

\appendix

\section{Computing $F_{3}$}\label{sec:app}

In this appendix we compute the term $F_{3}$ (see Proposition \ref{prop:LSred} and the subsequent discussion) at a metric $g_{\infty}$ with constant scalar curvature. We believe this computation is of certain interest since, as the reader may check from the sequel, the higher order polynomials $F_{p}$ for $p\geq 4$ cannot be determined explicitly since this would require stronger information on the reduction map $\Phi$ (or, equivalently, on $\Psi$) at the linearization point than we actually have according to Proposition \ref{prop:LSred}. 

Denote by $\left\langle \,\cdot\,,\,\cdot\,\right\rangle$ the $L^{2}(M,g_\infty)$-pairing and, without further discussion we refer to Section \ref{sec:def-prelim} for the notation concerning differentials and gradients. First, we will show that $F_{1}(v) = F_{2}(v) = 0$. To check this, notice that 
$DF(w)[v] = 
D\mathcal{Y}(\Psi(w))\left[D\Psi(w)[v]\right].$
Thus, $DF(0) = 0$ as $D\mathcal{Y}(1) = 0$ as $1$ is a critical point of the Yamabe functional (by assumption, $g_{\infty}\in \mathcal{CSC}_{1}$) and of course $\Psi(0)=1$. Therefore, $F_{1} = 0$. Similarly, 
$D^{2}F(w)[v,u] = 
D^{2}{\mathcal{Y}}(\Psi(w)) 
\left[D\Psi(w)[u] , D\Psi(w)[v]\right]  + 
\left\langle D{\mathcal{Y}}(\Psi(w)) , D^{2}\Psi(w)[v,u]  \right\rangle.
$
When setting $w=0$, $\Psi(0)=1, D\Psi(0)=\hbox{\rm Id}$, and
\begin{align*}
D^{2}F(0)[v,u] & = 
D^{2}{\mathcal{Y}}(1)
[u,v]  + 
\left\langle D{\mathcal{Y}}(1) , D^{2}\Psi(0)[v,u]  \right\rangle\\
& =
-2(N-2)
\langle \cL_{\infty} u,v\rangle  + 
\left\langle D{\mathcal{Y}}(1) , D^{2}\Psi(0)[v,u]  \right\rangle.
\end{align*}
As before, the second term vanishes. The first term vanishes because $v$ is in the kernel of the linearization of $\cL_{\infty}$, by assumption.

As observed in \cite[Remark 1.19]{AdamsSimon}, one may explicitly compute $F_{3}$, without explicit knowledge of $\Psi$ (and this is what typically makes $AS_3$ simpler to check than $AS_p$ with $p>3$ in explicit examples). We will use this observation and check that to compute $D^{3}F(0)$, one may in fact compute $D^{3}\tilde F(0)$ where $\tilde F:\Lambda_{0}\to \RR$ is defined by $\tilde F(v) = \mathcal{Y}(1+v)$. We first compute $D^{3}F$:
\begin{equation*}\begin{split}
D^{3}F(w)[v,u,z] & = 
D^{3}{\mathcal Y}(\Psi(w))
[D\Psi(w)[v],D\Psi(w)[u],D\Psi(w)[z]] \\
 & \qquad 
 + 
D^{2}{\mathcal Y}(\Psi(w))
[D^{2}\Psi(w)[u,z],D\Psi(w)[v]]
\\
& \qquad 
+D^{2}{\mathcal Y}(\Psi(w))
[D\Psi(w)[u],D^{2}\Psi(w)[v,z]]
\\
& \qquad 
+ 
D^{2}{\mathcal Y}(\Psi(w))
[D\Psi(w)[z],D^{2}\Psi(w)[v,u]]
\\
& \qquad 
+ \left\langle D{\mathcal Y}(\Psi(w)), 
D^{3}\Psi(w)[v,u,z]\right\rangle.
\end{split} \end{equation*}
Setting $w=0$, and using similar considerations as before (in particular noting that $D^{2}{\mathcal Y}(1)[\cdot]$ is self-adjoint),
we obtain
$D^{3}F(0)[v,u,z] = D^{3}{\mathcal Y} (1)[v,u,z].
$ 
Performing the same computation for $D^{3}\tilde F(0)$ yields the same result. Next, we compute $D^{3}\tilde F(0)$. Recall 
from Section 2 that the differential of the Yamabe energy
 is
$\frac 1 2 D\cY(w)[v] =  
\int_{M}\left[ - (N+2) \Delta_{g_{\infty}} w + R_{g_{\infty}} w  - r_{w^{N-2}g_{\infty}} w^{N-1} \right] v dV_{g_{\infty}}.  
$ 
The first two terms are linear in $w$, so when computing the third derivative of $\cY$ at $1$, they will vanish. Let us then concentrate on the third term.
Because $r_{w^{N-2}g_{\infty}}=\cY(w)$ we have already shown that the first and second directional derivatives of this expression
in directions in $\Lambda_0$ vanish at $w=1$.
Hence, we see that the following expression holds:
\begin{equation*}
D^3\cY(1)[u,z,v] =  
- 2(N-1)(N-2) r_{g_{\infty}}
\int_{M} uzv dV_{g_\infty},
\end{equation*}
for $u,z,v\in \Lambda_{0}$, proving \eqref{eq:F3-comp}.

In this final paragraph (contrary to the rest of this section) 
we will use the space-time $C^{k,\alpha}$ norms on an interval $(t,t+1)\times M$, as in Section \ref{sec:slow-flow}. One may observe that by repeating the argument used above for $w$ such that 
$\Vert w-1\Vert_{C^{2,\alpha}}<1$ it is clear that the $C^{0,\alpha}$ norm 
of
$D^3\cY(w)[v,u]$,
regarded (via the $L^{2}(M,g_{\infty})$ pairing) as a function on $M$,
can be
bounded by a uniform constant times the 
$C^{2,\alpha}$ norm of $u$ times that of $v$. More precisely,
\begin{equation}\label{eq:app-D3bds}
\Vert D^3\cY(w)[u,v] \Vert_{C^{0,\alpha}}
\leq
C 
\Vert u\Vert_{C^{2,\alpha}}\Vert v\Vert_{C^{2,\alpha}}
\end{equation}
for some uniform $C>0$.
Furthermore, for $w_1,w_2$ such that $\Vert w_i-1\Vert_{C^{2,\alpha}}<1$ (for $i=1,2$), we have
\begin{equation*}
\Vert D^3\cY(w_1)[v,v]-
D^3\cY(w_2)[u,u] \Vert_{C^{0,\alpha}}
\leq
C (\Vert w_1\Vert_{C^{2,\alpha}}+\Vert w_2\Vert _{C^{2,\alpha}})(\Vert u\Vert_{C^{2,\alpha}} + \Vert v\Vert_{C^{2,\alpha}})
\Vert u-v\Vert_{C^{2,\alpha}}
\end{equation*}
for some uniform $C>0$.
These facts are used in the proof of Lemma \ref{lemm:DYu-est} and Proposition \ref{prop:slow-decay-rewrite-two-comp}.

\section{Monotonicity of the period function}\label{app:mon-period}

Here we review the proof of \cite[Lemma 1.2]{BidautVeronBouhar} in
our special setting.
Recalling \eqref{HamEq}, define the ``potential energy" 
$U(u)=H/2-v^2=
\big(\frac2{N-2}\big)^2(u^N-u^2)$. Its absolute minimum in the range $(0,1)$ is
attained at $u=u_0$.
\def\al{\alpha}\def\be{\beta}\def\la{\lambda}\def\del{\partial}
\def\ra{\rightarrow}
Denote by $\lambda_\be(t)=(u(t),v(t))$ the solution of \eqref{eq:YP-system}
with $\la_\be(0)=(u_0,\beta)\in\Omega$, with
$\be\in[0,\sqrt{-U(u_0)})$ (with $\la_0=(u_0,0)$). This solution intersects the $u$-axis at
exactly two points that we denote by $(z_-(\be),0)$ and $(z_+(\be),0)$
with $z_-(\be)<z_+(\be)$. Since $v=du/dt$, $dt=du/v=du/\sqrt{H/2-U}$,
so the half-period $\tau(\be)/2$ of $\lambda_\be(t)$ is given 
by $\int_{u_0}^{z_+(\be)}du/\sqrt{H(\be)/2-U(u)}-\int_{u_0}^{z_-(\be)}du/\sqrt{H(\be)/2-U(u)}$,
where $H(\be)=2\be^2+2U(u_0)$.
Note that $U(z_\pm(\be))=H(\be)/2=\be^2+U(u_0)$, so differentiation in $\be$
gives $\frac{\del U}{\del u}(z_\pm(\be))z_\pm'(\be)=2\be$. Thus,
setting $a=\sqrt{U(u)-U(u_0)}/\be$, gives
$\tau(\be)/2=\int_{0}^1(z_+-z_-)'(\be t)\frac{dt}{\sqrt{1-t^2}}.$
The advantage of this formula is its simple dependence on $\be$: it
suffices to show now that $z_+-z_-$ is convex
in $\be\in(0,\sqrt{-U(u_0)})$ (note $-U(u_0)=\frac1{N-2}\left(\frac2N\right)^{\frac N{N-2}}$).
Geometrically, this means that the ``width" of the domains enclosed by the image
of $\la_\be$ is convex as a function of their ``height" $2\be$.
Differentiation in $\be$ yields
$
z_{\pm}''(\be)/2=\frac{U'^2-2U''(U-U(u_0))}{U'^3}(z_\pm(\be)).
$
L'H\^opital's rule applied twice immediately gives (using that
$\lim_{\be\ra0}z_\pm(\be)=u_0$ and $U'(u_0)=0$) 
$\lim_{\be\ra0}z''_\pm(\be)/2=-U'''(u_0)/3U''^2(u_0)=:A<0$.
The convexity claim follows if $z''_+(\be)/2\ge A\ge z''_-(\be)/2$,
for $\be\in(0,\sqrt{-U(u_0)})$. Since the sign of $U'(z_\pm)$ (which
is the sign of the denominator of $z_\pm''$)
is $\pm$, both inequalities follow if
$F(u):=
U'^2-2U''(U-U(u_0))-AU'^3\ge0
$
on $u\in(0,1)$. 
Now,
$U''(u)=
\big(\frac2{N-2}\big)^2(N(N-1)u^{N-2}-2)$
is negative on $\big(0,\big(\frac2{N(N-1)}\big)^{\frac1{N-2}}\big)\subset(0,1)$;
so, as $U'''>0$, $F'=-2U'''(U-U(u_0))-3AU'^2U''\le0$, i.e.,
$F$ does not increase in that range. Thus, is suffices to show
that $F\ge0$ in $\big(\big(\frac2{N(N-1)}\big)^{\frac1{N-2}},1\big)$. In that
regime (where $U''>0$), consider the function $H:=F/U''$, and compute
$H'=\frac{U'^2U'''}{U''^2}[A(U'-3U''^2/U''')-1]$. Denote the expression in the brackets
by $K$ and note the sign of $H$ equals the sign of $K$. Now
$K'=A\Big(U''-\frac{6U''U'''^2-3U''''U''^2}{U'''^2}\Big),$
or
$K'=\frac{AU''}{U'''^2}\Big(-5U'''^2+3U''''U''\Big),$
whose sign is opposite the sign of the expression in the paranthesis,
that we denote by $L$. But $L=-\frac92(U'')^{\frac83}\big((U'')^{-\frac23}\big)''$,
and $(U'')^{-\frac23}$ is seen to be convex on $\big(\big(\frac2{N(N-1)}\big)^{\frac1{N-2}},1\big)$; 
thus $K'\ge0$ there (as $U''>0$ there).
Now, $K$ vanishes at $u_0$, so $K\ge0$ and $H'\ge0$ on $(u_0,1)$. But 
$F(u_0)=H(u_0)=0$, so $H\ge0$ and $F\ge0$ on $(u_0,1)$. Further, $K$ must be
negative on $\big(\big(\frac2{N(N-1)}\big)^{\frac1{N-2}},u_0\big)$ (as $K'\ge0$ on
$(\frac2{N(N-1)}\big)^{\frac1{N-2}},1)$ while $K(u_0)=0$).
Thus $H'\le0$ on $\big(\big(\frac2{N(N-1)}\big)^{\frac1{N-2}},u_0\big)$, so $H$
is nonincreasing there; but $H(u_0)=0$, so we must have $H\ge0$
also on $\big(\big(\frac2{N(N-1)}\big)^{\frac1{N-2}},u_0\big)$. In conclusion,
$F\ge0$ on $(0,1)$, as desired.



\providecommand{\bysame}{\leavevmode\hbox to3em{\hrulefill}\thinspace}
\providecommand{\MR}{\relax\ifhmode\unskip\space\fi MR }
\providecommand{\MRhref}[2]{%
  \href{http://www.ams.org/mathscinet-getitem?mr=#1}{#2}
}
\providecommand{\href}[2]{#2}

\end{document}